\documentclass[12pt]{amsart}

\usepackage{amsmath,amsfonts,amsthm,amssymb}
\usepackage[margin=1in]{geometry}
\usepackage{tikz}
\usepackage{pgfplots}
\usepackage{comment}
\usepackage{xypic}
\usepackage{graphicx}
\usepackage[bookmarks,colorlinks,citecolor=blue]{hyperref} 
\usepackage{tikz}

\newlength{\defbaselineskip} 
\theoremstyle{theorem}
\newtheorem{thm}{Theorem}[section]
\newtheorem{cor}[thm]{Corollary}
\newtheorem{lema}[thm]{Lemma}
\newtheorem{obs}[thm]{Proposition}

\newtheorem{question}[thm]{Question}

\newtheorem{pr}{Algorithm}
\newtheorem{fact}[thm]{Fact}
\theoremstyle{definition} 
\theoremstyle{definition} \newtheorem{ex}{Example}[section] %

 \numberwithin{equation}{section}
 
 \theoremstyle{definition}
\newtheorem{con}[thm]{Conjecture}
\newtheorem{df}[thm]{Definition}
\newtheorem{exm}[thm]{Example}
\newtheorem{rem}[thm]{Remark}

\def\P{\mathbb{P}}
\def\r{\mathbb{R}}

\def\Z{\mathbb{Z}}
\def\N{\mathbb{N}}

\def\C{\mathbb{C}}
\def\R{\mathbb{R}}
\def\Q{\mathbb{Q}}

\def\o{\mathcal{O}}

\DeclareMathOperator{\Sec}{Sec}
\DeclareMathOperator{\im}{Im}
\DeclareMathOperator{\conv}{conv}

\def\CC{\mathbb{C}}
\def\ob{\begin{obs}}
\def\kob{\end{obs}}
\def\dow{\begin{proof}}
\def\kdow{\end{proof}}

\def\tw{\begin{thm}}
\def\ktw{\end{thm}}
\def\hip{\begin{con}}
\def\khip{\end{con}}
\def\lem{\begin{lema}}
\def\klem{\end{lema}}
\def\ex{\begin{exm}}
\def\prog{\begin{pr}}
\def\kprog{\end{pr}}
\def\wn{\begin{cor}}
\def\kwn{\end{cor}}
\def\uwa{\begin{rem}}
\def\kuwa{\end{rem}}
\def\kex{\end{exm}}
\def\dfi{\begin{df}}
\def\kdfi{\end{df}}
\xyoption{line}
\def\fa{\begin{fact}}
\def\kfa{\end{fact}}

\title{Secant varieties of toric varieties arising from simplicial complexes}
\author{M. Azeem Khadam}
\address{Abdus Salam School of Mathematical Sciences, GC University Lahore, Pakistan}
\email{azeem.khadam@sms.edu.pk}
\author{Mateusz Micha\l ek}
\address{Max Planck Institute for Mathematics in the Sciences, 04103 Leipzig, Germany and \
Aalto University, Espoo, Finland and
Polish Academy of Sciences, Warsaw, Poland}
\email{wajcha2@poczta.onet.pl}
\author{Piotr Zwiernik}
\address{Department of Economics and Business, Universitat Pompeu Fabra, Barcelona, Spain}
\email{piotr.zwiernik@upf.edu}
\thanks{AK is supported by ASSMS, GC University Lahore under a postdoctoral fellowship. PZ was supported from the Spanish Government grants (RYC-2017-22544,PGC2018-101643-B-I00,SEV-2015-0563), and Ayudas Fundaci\'on BBVA a Equipos de Investigaci\'on Cientifica 2017.}

\begin{document}
\maketitle
\begin{abstract}
Motivated by the study of the secant variety of the Segre-Veronese variety we propose a general framework to analyze properties of the secant varieties of toric embeddings of affine spaces  defined by simplicial complexes. We prove that every such secant is toric, which gives a way to use combinatorial tools to study singularities. We focus on the Segre-Veronese variety for which we completely classify their secants that give Gorenstein or $\Q$-Gorenstein varieties. We conclude providing the explicit description of the singular locus.  
\end{abstract}

\section{Introduction}
The study of low rank tensors is a fundamental topic linking algebra and geometry with an increasing number of applications, for example, in quantum physics and machine learning. Two of the most basic projective algebraic varieties: \emph{Segre product} and \emph{Veronese embedding} are exactly the geometric loci of tensors, or symmetric tensors, of rank one. These varieties play a fundamental role in algebraic geometry and are used as a starting point for other interesting constructions. 

For rank greater than one, the best understood is the case of two dimensional tensors, that is, matrices. One of the first theorems we learn in linear algebra is that the set $S_{a,b,r}$ of $a\times b$ matrices of rank at most $r$ is defined by the vanishing of minors of size $r+1$. This realizes $S_{a,b,r}$ as an algebraic variety \cite[Lecture 9]{harris2013algebraic}. Even this simple case has interesting geometry: $S_{a,b,r}$ is \emph{singular} along $S_{a,b,r-1}$. The nature of the singularities has been a topic of intensive studies \cite[Chapter 6]{bruns2006determinantal}. Classical results tell us that $S_{a,b,r}$ is always normal, Cohen-Macaulay, and it is Gorenstein if and only if $a=b$  \cite[Theorem 6.3, Corollary 8.9]{bruns2006determinantal}, \cite[Theorem 5.5.6]{svanes1974coherent}. Similar results hold for symmetric matrices.

For higher dimensional tensors, in spite of a large amount of work \cite{landsberg2012tensors, landsberg2017geometry}, \cite[Chapter 9]{ourbook} the situation is much less clear. First, tensors of rank at most $r$ do not have to form a closed set. Hence, the natural geometric object to consider is the closure of that set, known as the $r$-th \emph{secant variety}. Some secant varieties of Segre products were proved to be arithmetically Cohen-Macaulay \cite{landsberg2007ideals,oeding2016equations} which, in general, is only a conjecture \cite{oeding2016all}. One of the most basic questions is about the dimension; see, for example, \cite{abo2012new,catalisano2005higher,laface2013secant}.

Our main focus in this article is on the 2nd secant varieties, classically referred to as \emph{the secant varieties}, of the Segre-Veronese varieties, which have been intensively studied over the last two decades. The secant varieties of Segre-Veronese varieties are known to be normal \cite[Theorem 2.2]{vermeire2009singularities}. The tangential variety of the Segre product is also always normal \cite[Proposition 8.5]{MOZ}, however this is no longer the case for the Segre-Veronese \cite[Example 2.21]{michalek2018flexible}.
 More refined questions about the geometry of secant varieties were possible to address only in special cases. One of the breakthroughs was made by Raicu who proved, after partial attempts \cite{allman2008phylogenetic,landsberg2009secant,landsberg2004ideals,landsberg2007ideals}, a conjecture by Garcia, Stillman, and Sturmfels \cite{garcia2005algebraic} providing a complete list of generators of the ideal of the secant variety \cite{raicu2012secant}. This description is given in terms of minors of flattening matrices and the methods rely on representation theory. Analogous result for the tangential variety was obtained in \cite{oeding2014tangential}. 

A large part of the research on the geometry of secant varieties has been motivated by applications in an increasing number of fields such as geometric complexity theory \cite{landsberg2017geometry}, algebraic statistics \cite{AStensors,pwz}, and machine learning \cite{cichocki2014era} where understanding geometry of low rank tensors gives insight into designing and analysis of numerical algorithms used to find low rank approximations \cite{hackbusch2012tensor, qi2019complex}. Segre-Veronese varieties also appear in quantum statistics \cite{cinardi2019quantum} and complex network geometry \cite{bianconi2015complex,zuev2015exponential}. 

Despite intense studies, many questions regarding the secant varieties of the Segre-Veronese varieties remained open. So far, even for border rank two, the complete description of the singular locus of Segre-Veronese varieties was not known. This challenge was one of the main motivations of our article. In Corollary \ref{cSLSV} we provide a complete description of the singular locus for the secant variety of any Segre-Veronese variety. This extends the results for the Segre \cite{MOZ} and the Veronese \cite{kanev1999chordal} cases. (For secant varieties of the Veronese the singular locus was described in small cases in \cite{han2018singularities}.) We also provide, in Theorem \ref{GSV} and Theorem \ref{QGSV}, a complete classification of the secant varieties of Segre-Veronese varieties that are ($\Q-$) Gorenstein, extending previous results for matrices and Segre products.

To study secants of rank one tensor varieties it is customary to use techniques from representation theory. We further develop methods initiated in \cite{MOZ}, where a completely different set of techniques inspired by probability was proposed. The main idea is to treat points of the variety as (formal) probability distributions and apply methods from algebraic statistics. A change of coordinates, inspired by cumulants, leads to new structures on secant varieties. Although application of `cumulant methods' in algebraic geometry is very nonstandard, it has been previously applied with success: The singular locus of the secant and the tangential varieties of the Segre product was studied in \cite{MOZ,sturmfels2013binary}. It was also used to prove flexibility of secant varieties of Segre-Veronese \cite{michalek2018flexible} and to study secant varieties of Grassmannians \cite{manivel2015secants}. Some generalizations of these techniques have also been explored in the context of other classical varieties \cite{ciliberto2016cremona}.

Although studying low rank tensors was our main motivation, the setting we propose is much more general. To any (labelled) simplicial complex we associate an embedding of an affine space. Only special simplicial complexes correspond to Segre-Veronese varieties. It turns out that secant and tangential varieties of such embeddings are always \emph{toric varieties}; see Theorem \ref{main} and Theorem \ref{main-tangent}. This allows us to apply the powerful machinery of toric geometry to study secant varieties of Segre-Veronese varieties --- which are not toric.

An interesting special case of the simplicial toric embedding, that is when the simplicial complex is a graph, has already appeared in the literature. Here it turns out that the secant and tangential varieties coincide and are isomorphic to a product of an affine space and a toric variety associated to a graph in the sense of Hibi and Ohsugi \cite{Hibi}. These toric varieties have also been intensively studied \cite{ohsugi2006special,villarreal1990cohen}. Moreover, toric models correspond in algebraic statistics to discrete exponential families \cite[Section 6.2]{Seth}. Their secants correspond to mixture models \cite[Chapters 14 and 16]{Seth}. In this context, the singular locus also plays an important role \cite[Example 14.1.9]{Seth}. Further, secant and tangential varieties of toric varieties are an object of study in pure mathematics, see for example~\cite{cox2007secant, raicu2012secant}.


This article is organized as follows. Section~\ref{sec:simplicial} deals with toric varieties and simplicial embeddings as well as simplicial cumulants. We prove in this section that the secant variety of the simplicial toric embedding is the product of an affine space and a toric variety; see Theorem \ref{main}. In Section~\ref{sec:Segre-Veronese}, we discuss the polytope of the above mentioned toric variety and prove that it is normal; see Lemma \ref{lem:Pnormal}. Moreover, we also present the complete facet description of the polytope; see Proposition \ref{pSVF}. In Section~\ref{sec:singularity}, by using the polytope description from previous section, we present the complete classification of ($\Q-)$Gorenstein property of the secant varieties of the Segre-Veronese variety; see Theorem \ref{GSV} and Theorem \ref{QGSV}. In Section~\ref{sec:locus} we conclude this article by providing a complete description of the singular locus of the secant of the Segre-Veronese variety; see Corollary \ref{cSLSV}.


\section{Secant varieties, simplicial complexes and toric geometry}\label{sec:simplicial}

\subsection{Toric varieties and simplicial embeddings}

Let $\mathbf x=(x_1,\ldots,x_N)$ and $\mathbf t=(t_1,\ldots,t_n)$, and fix a subset $\mathcal C=\{\mathbf c_1,\ldots,\mathbf c_N\}$ of $\mathbb N^n$, where $\N$ denotes the set of nonnegative integers. Each vector $\mathbf c_i$ is identified with a monomial $\mathbf t^{\mathbf c_i}:=t_1^{c_{i1}}\cdots t_n^{c_{in}}$. The set $\mathcal C$ defines a map $e_{\mathcal C}$ from $\C^n$ to $\C^N$ where $x_{i}=\mathbf t^{\mathbf c_i}$ for $1\leq i\leq N$. The closure of the image of this map $V_{\mathcal C}:=\overline{e_{\mathcal C}(\C^n)}$ is called an \emph{affine toric variety}. We note that this differs from the classical definition of a toric variety, where, in addition, the variety needs to be normal; see \cite[Chapter 13]{sturmfelsks} for discussion.

A \emph{simplicial complex} $\Delta$ on the vertex set $\{1,\ldots,n\}$ is a collection of subsets, called \emph{simplices}, closed under taking subsets, that is, if $\sigma\in \Delta$ and $\tau\subset \sigma$ then $\tau\in \Delta$. A simplex $\sigma\in \Delta$ of cardinality $|\sigma|=i+1$ has \emph{dimension} $\dim(\sigma)=i$. In this paper we allow vertices to have repeated labels. In this case $\{1,\ldots,n\}$ always refers to the labelling set of $\Delta$ rather than its vertex set; see Figure~\ref{fig:simplicial1} for an example.

By the standard construction a simplicial complex defines an affine toric variety.
Let $\Delta$ be a simplicial complex with vertices labelled by variables $\mathbf t=(t_1,\dots,t_n)$ (with possible repetitions). Suppose $\Delta$ contains $N$ distinct simplices. Then $\Delta$ induces an embedding $e_\Delta:\CC^n\rightarrow\CC^N$, where the coordinates of the codomain are indexed by $\sigma\in \Delta$, by 
$$	\mathbf t\mapsto \mathbf x=(x_\sigma)_{\sigma\in \Delta},\qquad x_\sigma=\prod_{i\in \sigma} t_i.$$
 By convention, the monomial corresponding to the empty set is $x_\emptyset=1$. If two simplices have exactly the same labels, as multisets, we may identify them.
 We define the variety $V_\Delta\;:=\;\overline{e_\Delta(\C^n)}$.  Denote by $\Delta_{\geq 2}$ the set of simplices in $\Delta$ of dimension at least one. The toric variety $T_\Delta$ associated with the embedding $e_\Delta$ is the affine toric variety in $\C^{N-n-1}$ obtained as the closure of the projection of $e_\Delta$ to the coordinates $x_\sigma$ with $\sigma\in \Delta_{\geq 2}$. In this article both $V_\Delta$ and $T_\Delta$ have an important role to play.

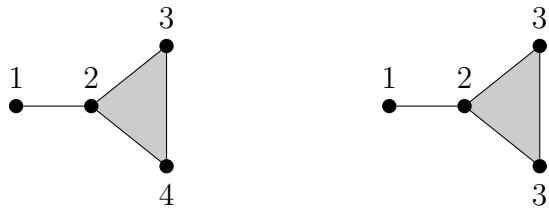
\begin{figure}[htp!]
\begin{tikzpicture}[scale=2, vertices/.style={draw, fill=black, circle, inner sep=0pt,minimum size=5pt}]
\filldraw[fill=black!20, draw=black] (-0.5,0.5)--(0,.1)--(0,.9)--cycle;
\node[vertices] (3) at (-1,0.5) [label=above:$1$]{};
\node[vertices] (5) at (-0.5,0.5) [label=above:$2$]{};
\node[vertices] (6) at (0,0.1) [label=below:$4$]{};
\node[vertices] (7) at (0,.9) [label=above:$3$]{};
\draw (3)--(5);
\end{tikzpicture}\qquad\qquad\qquad	\begin{tikzpicture}[scale=2, vertices/.style={draw, fill=black, circle, inner sep=0pt,minimum size=5pt}]
\filldraw[fill=black!20, draw=black] (-0.5,0.5)--(0,.1)--(0,.9)--cycle;
\node[vertices] (3) at (-1,0.5) [label=above:$1$]{};
\node[vertices] (5) at (-0.5,0.5) [label=above:$2$]{};
\node[vertices] (6) at (0,0.1) [label=below:$3$]{};
\node[vertices] (7) at (0,.9) [label=above:$3$]{};
\draw (3)--(5);
\end{tikzpicture}
\caption{Two simplicial complexes with four vertices over $\{1,2,3,4\}$ and over $\{1,2,3\}$ respectively.}	\label{fig:simplicial1}
\end{figure}

\begin{exm}\label{ex:basicex1}Let $\Delta$ be a simplicial complex with generators $\{1,2\}$, $\{2,3,4\}$; c.f. Figure~\ref{fig:simplicial1} on the left. Then $e_\Delta$ is given by 
$$
(t_1,t_2,t_3,t_4)\;\;\mapsto \;\;(1,t_1,t_2,t_3,t_4,\underline{t_1t_2,t_2t_3,t_2t_4,t_3t_4,t_2t_3t_4}).
$$
The underlined part corresponds to $x_\sigma$ for $\sigma\in \Delta_{\geq 2}$. The associated toric variety $T_\Delta$ is given by a single equation $x^{2}_{234}-x_{23}x_{24}x_{34}$. Specializing now to $t_3=t_4$ gives an example of a simplicial complex with vertices labelled with repetitions; see Figure~\ref{fig:simplicial1} on the right. Here $e_\Delta$ is given by 
$$
(t_1,t_2,t_3)\;\;\mapsto \;\;(1,t_1,t_2,t_3,\underline{t_1t_2,t_2t_3,t_3^2,t_2t_3^2}).
$$
The equation of $T_\Delta$ is $x_{233}^2-x_{23}^2x_{33}$. Note that in this example, the number $N$ of distinct simplices dropped by two. 
\end{exm}

\begin{rem}
	Simplicial toric embeddings have appeared in quantum statistics. For example, consider equations (6) and (8) in \cite{cinardi2019quantum}. Here $x_\sigma$ is the associated fitness of a particular simplex in $\Delta$ with $t_i=e^{-\beta \epsilon_i}$, where $\beta$ is the inverse temperature and $\epsilon_i$ is the energy of a particular node. 
\end{rem}

A simplicial complex is a \emph{graph} if all of its simplices have dimension at most one. When $\Delta$ is a graph, then $e_\Delta(\CC^n)$ is the graph of the map parameterizing the toric varieties defined in \cite{Hibi,villarreal1990cohen}. In this case, as we will see in Theorem \ref{main}, the secant variety of $V_\Delta$ is isomorphic to the product of an affine space and the toric variety $T_\Delta$ associated to the graph.

\begin{exm} 
Consider the following graph $G$.
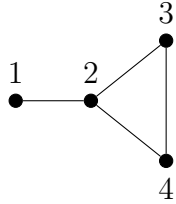
\begin{figure}[htp!]
\begin{tikzpicture}[scale=2, vertices/.style={draw, fill=black, circle, inner sep=0pt,minimum size=5pt}]
\node[vertices] (3) at (-1,0.5) [label=above:$1$]{};
\node[vertices] (5) at (-0.5,0.5) [label=above:$2$]{};
\node[vertices] (6) at (0,0.1) [label=below:$4$]{};
\node[vertices] (7) at (0,.9) [label=above:$3$]{};
\draw (3)--(5);
\draw (5)--(6);
\draw (5)--(7);
\draw (6)--(7);
\end{tikzpicture}
\end{figure}

\noindent We have the associated map:
$$e_G:\CC^4\ni(t_1,t_2,t_3,t_4)\mapsto(1,t_1,t_2,t_3,t_4,t_1t_2,t_2t_3,
t_2t_4,t_3t_4).$$
The image $e_G(\CC^4)$ is the graph of the function $(t_1,t_2,t_3,t_4)\mapsto (t_1t_2,t_2t_3,
t_2t_4,t_3t_4)$
parameterizing the toric variety associated to $G$; see  \cite{Hibi}.
\end{exm}


\subsection{Simplicial cumulants}

In what follows, for simplicity, we assume that $\Delta$ is connected. The coordinates of the ambient space $\CC^N$ will be denoted by $x_\sigma$ for $\sigma\in \Delta$. In the special case, when $\sigma$ is of dimension zero and is given by a vertex $i$, we will often write $\sigma=i$ instead of $\sigma=\{i\}$. 
We define an automorphism of $\CC^N$ given by
$$y_\sigma\;=\;x_\sigma,\qquad \mbox{if }\dim (\sigma)\leq  0,$$
and otherwise, for all $\sigma\in \Delta_{\geq 2}$
$$y_\sigma\;=\;\sum_{\sigma'\subseteq \sigma} (-1)^{\dim \sigma+\dim \sigma'}x_{\sigma'}\prod_{i\in \sigma\setminus \sigma'}x_{i},$$
where the sum is taken over all subsimplices $\sigma'$ of $\sigma$ including the empty simplex, and the product is taken over vertices $i\in \sigma\setminus \sigma'$.
\begin{exm}
 For the first simplicial complex in Example~\ref{ex:basicex1}, we have $y_{ij}=x_{ij}-x_ix_j$ for all edges $\{i,j\}\in \Delta$ and 
 $$
 y_{234}\;=\;x_{234}-x_{2}x_{34}-x_{3}x_{24}-x_{4}x_{23}+2x_2x_3x_4.
 $$
 For the second simplicial complex in this example, the formulas for $y_\sigma$ are obtained by replacing $4$ with $3$. So for example, the simplex $\sigma=\{2,3,3\}$ gives $$y_{233}=x_{233}-x_2x_{33}-2x_3x_{23}+2x_2x_3^2.$$\end{exm}

Order the vertices of the simplicial complex $\Delta$ according to their labelling with the variables $x_1,\dots,x_n$. If a few vertices have the same label, then we order them arbitrarily. For example, the simplicial complex on the right in Figure~\ref{fig:simplicial1} has a natural numbering of its four vertices with $\{1,2,3,4\}$. In general, this induces a linear order on the vertices of all simplices in $\Delta$.
\begin{df}[Thick interval partition]
Fix a linearly ordered set $B$. We say that $B=\bigcup_{i=1}^k B_i$ is a \emph{thick interval partition} if and only if the following conditions are satisfied:
\begin{enumerate}
\item $B_i$ is disjoint from $B_j$ for $i\neq j$,
\item if $i<j$ then for any $x\in B_i$, $y\in B_j$ we have $x<y$, and
\item the cardinality of any $B_i$ is at least $2$.
\end{enumerate}
The set of thick interval partitions of $B$ is denoted by $IP(B)$. For an interval partition $g\in IP(B)$ we denote by $|g|$ the number $k$ of sets $B_i$.
\end{df}
We define the following automorphism of $\CC^N$:
$$z_\sigma\;=\;y_\sigma,\qquad \mbox{if }\dim (\sigma)\leq  0,$$
and otherwise, for all $\sigma\in \Delta_{\geq 2}$
$$z_\sigma\;=\;\sum_{g\in IP(\sigma)}(-1)^{|g|+1}\prod_{\sigma\in g}y_{\sigma}.$$
The change of coordinates from $(x_\sigma)$ to $(z_\sigma)$ generalises secant cumulants defined in \cite{MOZ}; see \cite{zwiernik2012cumulants} for a more detailed analysis of related constructions. For a fixed simplicial complex $\Delta$ we call $(z_\sigma)$ \emph{simplicial cumulants}. 

\begin{rem}
	One of the important properties of simplicial cumulants is that the image of $V_\Delta$ in this coordinate system is contained in the affine space given by $z_\sigma=0$ for all $\sigma\in \Delta_{\geq 2}$. We refer to \cite{ciliberto2016cremona} for a more general discussion.
\end{rem}

Define an affine variety $\sigma_2(V_\Delta)$, that is, the secant variety of the embedding of the affine space $\CC^n$ given by $\Delta$. The variety $\sigma_2(V_\Delta)$ is parameterized by $2n+1$ parameters: $\pi$ and $u_i,t_i$ for $i=1,\ldots,n$, explicitly:
 $$x_\sigma\;=\;\pi\prod_{i\in\sigma}t_i+(1-\pi)\prod_{i\in\sigma}u_i\qquad\mbox{for all }{\sigma\in \Delta}.$$

All defined changes of coordinates are hierarchical in a sense that the formulas for $y_\sigma$ and $z_\sigma$ depend only on $x_{\sigma'}$ and $y_{\sigma'}$ for $\sigma'\subseteq \sigma$ respectively. In consequence the following result generalizes \cite[Lemma 3.1]{MOZ}.
\begin{lema}\label{lem:sec}The variety $\sigma_2(V_\Delta)$ in the coordinate system given by the simplicial cumulants is the Zariski closure of the image of the parameterization given by:
$$
\begin{array}{ll}
z_v\;=\;\pi t_v+(1-\pi) u_v & \mbox{if }\dim(v)=0,\\
z_\sigma\;=\;\pi(1-\pi)(1-2\pi)^{\dim{\sigma}-1}\prod_{i\in \sigma}(t_i-u_i)& \mbox{if }\dim(\sigma)\geq 1.
\end{array}
$$
%
\end{lema}
\begin{proof}
We regard $x_\sigma$, $y_\sigma$ and $z_\sigma$ as polynomials in $\pi,t_i,u_i$. We have to prove that $z_\sigma$ has the form given in the lemma. For $z_v$ this is obvious, as $z_v=y_v=x_v$. In full generality, if $\Delta$ is labelled with repetitions, we may have $t_{v_1}= t_{v_2}$ for distinct vertices $v_1,v_2$ of $\Delta$. However, if we prove that $z_\sigma=\pi(1-\pi)(1-2\pi)^{\dim{\sigma}-1}\prod_{i\in \sigma}(u_i-t_i)$ holds, when all $u_v,t_v$ are distinct variables, then the statement will follow simply by specializing $u$'s and $t$'s. Hence, we may assume that to every $v\in \Delta$ we have associated two independent variables $u_v$ and $t_v$. 
Next we apply the methods introduced in \cite{MOZ}.

Step 1: We show that $u_v-t_v$ divides $y_\sigma$ for every $v\in\sigma$. In other words we prove that if $u_v=t_v$ then $y_\sigma=0$. Assume $u_v=t_v$ and $v\in\sigma'\subseteq\sigma$. Then also $x_v=u_v=t_v$ and we have:
$$x_{\sigma'}\prod_{j\in\sigma\setminus\sigma'}x_j=\left(\pi\prod_{i\in\sigma'}t_i+(1-\pi)\prod_{i\in\sigma'}u_i\right)\prod_{j\in\sigma\setminus\sigma'}x_j=$$
$$=\left(\pi\prod_{i\in\sigma'\setminus\{v\}}t_i+(1-\pi)\prod_{i\in\sigma'\setminus\{v\}}u_i\right)\prod_{j\in(\sigma\setminus\sigma')\cup\{v\}}x_j=x_{\sigma'\setminus \{v\}}\prod_{j\in(\sigma\setminus\sigma')\cup\{v\}}x_j.$$
For $v\in \sigma'\subseteq \sigma$ let $\sigma''=\sigma'\setminus\{v\}$. Therefore, the contributions of $x_{\sigma'}\prod_{j\in\sigma\setminus\sigma'}x_j$ and $x_{\sigma''}\prod_{j\in\sigma\setminus\sigma''}x_j$ to $y_\sigma$ cancel, which proves that $y_{\sigma}=0$.

Step 2: We show that $y_\sigma=\left(\pi(1-\pi)^{1+\dim\sigma}+(-\pi)^{1+\dim\sigma}(1-\pi)\right)\prod_{i\in\sigma}(t_i-u_i)$. By Step $1$, we know that $y_\sigma$ is divisible by $\prod_{i\in\sigma}(u_i-t_i)$. Further, as $y_\sigma$ is homogeneous of degree $d:=1+\dim\sigma$ in $u_i$, $t_i$, we know that it is of the form $P(\pi)\prod_{i\in\sigma}(u_i-t_i)$ for some polynomial $P(\pi)$. To determine $P(\pi)$ we substitute $u_i=0$ and $t_i=1$. Then $x_{\sigma'}=\pi$ for every $\sigma'\neq\emptyset$. We obtain:
$$y_\sigma=(-\pi)^d+\sum_{\emptyset\neq\sigma'\subset\sigma} (-1)^{\dim\sigma+\dim\sigma'}\pi^{d-\dim\sigma'}=$$
$$(-\pi)^d+(-1)^{\dim\sigma}\sum_{k=1}^{d}(-1)^{k+1}{{d}\choose{k}}\pi^{d-k+1}=(-\pi)^d+\pi((1-\pi)^{d}-(-\pi)^{d})=$$
$$=\left( \pi(1-\pi)^{d}+(-\pi)^{d}(1-\pi) \right)$$
Step 3: The proof of the final formula, due to Step 2, is just a computation on polynomials in one variable --- we leave it for the reader.
\end{proof}

Denote by $\widehat T_{\Delta}$ the affine cone over $T_{\Delta}$. The following result provides our main application of simplicial cumulants. 

\begin{thm}\label{main}
The secant variety $\sigma_2(V_\Delta)$ is isomorphic to the product of $\CC^n$ and the variety $\widehat T_{\Delta}$. 
\end{thm}
\begin{proof}
By Lemma \ref{lem:sec}, the variety $\sigma_2(V_\Delta)$ is isomorphic to the closure of the image of the map:
$$
\begin{array}{ll}
z_v=\pi t_v+(1-\pi) u_v & \mbox{if }\dim(v)=0,\\
z_\sigma=\pi(1-\pi)(1-2\pi)^{\dim{\sigma}-1}\prod_{i\in \sigma}(u_i-t_i) & \mbox{if }\dim(\sigma)\geq 1.
\end{array}
$$
Make a change of parameter variables: $t'_v=t_v$, $\pi'=\pi(1-\pi)/(1-2\pi)^2$ and $u'_v=(u_v-t_v)(1-2\pi)$. Then if $\dim \sigma>0$ the coordinate $z_\sigma$ becomes a monomial parameterizing $T_{\Delta}$ times $\pi'$. This gives the parametrization of $\widehat T_{\Delta}$. Further, $z_v$ is the only coordinate in which $t'_v$ appears, thus it is independent from the other coordinates.
\end{proof}
We define the \emph{tangential variety} of a variety $Y$ as the union of all tangent lines to $Y$. In a similar way as above one can prove the following theorem.
\begin{thm}\label{main-tangent}
The tangential variety of $V_\Delta$ is isomorphic to the product of $\CC^n$ and the variety $T_{\Delta}$. 
\end{thm}
\begin{rem}
If $\Delta$ is a graph then $T_{\Delta}=\widehat T_{\Delta}$ and it is exactly the toric variety associated to a graph, defined by Hibi and Ohsugi \cite{Hibi}. 
\end{rem}

We now discuss two extreme examples. The first one concerns the secant of the Veronese variety. 

\begin{exm}\label{ex:veronese}
Let $a,n\in \N$ be positive integers and $V$ a set of $a\cdot n$ vertices such that for each $i=1,\ldots,n$, exactly $a$ vertices are labelled by $x_i$. Let $\Delta$ be the simplicial complex on $V$ satisfying $\sigma\in \Delta$ if and only if $|\sigma\cap V|\leq a$. Then $V_\Delta$ is the (affine) $a$-th Veronese embedding. The tangential (resp.~secant) variety of this Veronese embedding is thus isomorphic to the product of $\CC^n$ and the toric variety (resp.~the affine cone over the toric variety) parameterized by all monomials of degree at least two and at most $a$.

\end{exm}

Our next example shows how Theorem~\ref{main} generalizes the main result of \cite{MOZ}.
\begin{exm}\label{ex:segre}
Consider the set $V=V_1\sqcup \dots \sqcup V_k$, where each $V_i$ is of cardinality $b_i$. We label all elements of $V$ with distinct variables. Let $\Delta$ be the simplicial complex on $V$ satisfying $\sigma\in \Delta$ if and only if $|\sigma\cap V_i|\leq 1$ for all $i=1,\dots, k$. The variety $V_\Delta$ is the (affine) Segre product $\CC^{b_1}\times\dots\times\CC^{b_k}$. Hence, the tangential (resp.~secant) variety of this Segre product is isomorphic to the product of $\CC^{b_1+\cdots+b_k}$ and the toric variety (resp.~the affine cone over the toric variety) parameterized by all monomials of degree at least two and at most one in each set of variables corresponding to each $\CC^{b_i}$. \end{exm}

One obvious way of applying Theorem~\ref{main} is to use toric geometry to study local properties of the secant (and tangential) varieties of toric embeddings of affine spaces given by simplicial complexes. In Sections \ref{sec:Segre-Veronese}-\ref{sec:locus}, we use this to study singularities of the secant variety of the Segre-Veronese embedding.

\section{The Segre-Veronese variety and its secant}\label{sec:Segre-Veronese}

A special case of the construction given in Section~\ref{sec:simplicial} will be used to study the Segre-Veronese variety. Fix $k\in \N$ a positive integer and $\mathbf a,\mathbf b\in \N^k$ where $\mathbf a=(a_1,\ldots,a_k), \mathbf b=(b_1,\ldots,b_k)$ such that $a_i, b_i$ are positive integers. Consider the vertex set $V=V_1\sqcup \cdots \sqcup V_k$, where each $V_i$ has $a_ib_i$ vertices such that for each $j=1,\ldots, b_i$, exactly $a_i$ vertices get labelled $t_{i,j}$. In this section $\Delta_{\rm SV}$ denotes a simplicial complex with vertex set $V$. A subset $\sigma$ of $V$ forms a simplex of $\Delta_{\rm SV}$ if and only if $|\sigma\cap V_i|\leq a_i$ for all $i=1,\ldots,k$.  In particular, if $k=1$ then we get the simplicial complex of Example~\ref{ex:veronese} and if $\mathbf a=(1,\ldots,1)$, we get the simplicial complex of Example~\ref{ex:segre}. 

If $\sigma\in \Delta_{\rm SV}$ then $\sigma=\sigma_1\sqcup\cdots\sqcup \sigma_k$, where each $\sigma_i$ is a multiset of labels $t_{i,j}$, where $|\sigma_i|\leq a_i$. Let $n=b_1+\cdots +b_k$ and $N$ be the number of simplices in $\Delta_{\rm SV}$. The toric embedding $e_{\Delta_{\rm SV}}:\C^n\to \C^N$ is given by 
$$
x_\sigma\;\;=\;\;\prod_{i=1}^k\prod_{j\in \sigma_i}t_{i,j}\qquad\mbox{for all}\quad\sigma\in \Delta_{\rm SV}.$$
The corresponding projective variety is obtained by introducing additional variables $t_{i,0}$ for $i=1,\ldots,k$ (the coordinates of each $\P^{b_i}$ are $(t_{i,0},\ldots,t_{i,b_i})$) and considering now a homogeneous parameterization $\P^{b_1}\times \cdots\times \P^{b_k}\to \P^{N}$
$$
x_\sigma\;\;=\;\;\prod_{i=1}^kt_{i,0}^{a_i-|\sigma_i|}\prod_{j\in \sigma_i}t_{i,j}\qquad\mbox{for all}\quad\sigma\in \Delta_{\rm SV}.$$
The image of this map is the \emph{Segre-Veronese variety} 
$$X\;:=\;v_{a_1}(\P^{b_1})\times\dots\times v_{a_k}(\P^{b_k}),$$ 
which is the embedding of the product $\P^{b_1}\times\dots\times\P^{b_k}$ given by the very ample line bundle $\o(a_1,\ldots,a_k)$. \textbf{From this point onward $\Delta_{\rm SV}$ will always be denoted by $\Delta$.}
\begin{rem}\label{rem:cover}
	The original affine variety $V_{\Delta}$ is isomorphic to the open subset of the Segre-Veronese variety obtained by setting $t_{i,0}\neq 0$ for all $i=1,\ldots,k$. This amounts to assuming $x_\emptyset\neq 0$. The Segre-Veronese variety can be covered by such varieties obtained by assuming that exactly one variable $t_{i,j}$ for each $i=1,\ldots,k$ is necessarily nonzero, or in other words, that a given coordinate $x_\sigma$ is nonzero.  
\end{rem}

\begin{exm}
	Suppose $\mathbf a=(2,2,1)$, $\mathbf b=(1,2,2)$. Then $V_1=\{t_{1,1},t_{1,1}\}$, $V_2=\{t_{2,1},t_{2,1},t_{2,2},t_{2,2}\}$, $V_3=\{t_{3,1},t_{3,2}\}$. The maximal simplices in $\Delta$ are 
	$$\{t_{1,1},t_{1,1},t_{2,1},t_{2,1},t_{3,1}\}\qquad \{t_{1,1},t_{1,1},t_{2,1},t_{2,2},t_{3,1}\}\qquad \{t_{1,1},t_{1,1},t_{2,2},t_{2,2},t_{3,1}\}$$
	$$\{t_{1,1},t_{1,1},t_{2,1},t_{2,1},t_{3,2}\}\qquad \{t_{1,1},t_{1,1},t_{2,1},t_{2,2},t_{3,2}\}\qquad \{t_{1,1},t_{1,1},t_{2,2},t_{2,2},t_{3,2}\}$$

	A lower-dimensional simplex $\sigma=\{t_{1,1},t_{1,1},t_{2,1},t_{2,2}\}$ gives the monomial $x_\sigma=t_{1,1}^2t_{2,1}t_{2,2}$, which can be more compactly written as $\mathbf t^{\mathbf c_\sigma}$ for $\mathbf c_\sigma=(2,1,1,0,0)$. The corresponding homogeneous version is  $x_\sigma=t_{1,1}^2t_{2,1}t_{2,2}t_{3,0}$.\end{exm}

In this and the following section we study the secant variety $\sigma_2(X)$ of the Segre-Veronese variety $X$. We rely on the following crucial observation. 
\begin{rem}\label{rem:covered}
	Consider the open subset of $\sigma_2(X)$ given by $\sigma_2(X)\cap \{x_\emptyset\neq 0\}$. On this subset $\sigma_2(X)$ is isomorphic to the secant of the affine variety $V_{\Delta}$. By Theorem~\ref{main} this (affine) secant variety is isomorphic to the product of $\C^n$ and the variety $\widehat T_\Delta$ associated to the simplicial complex $\Delta$. This means that $\sigma_2(X)$ can be covered by toric varieties, which is our main motivation to study the variety $\widehat T_{\Delta}$. 
\end{rem}

By following Remark~\ref{rem:covered}, we define the associated polytope $P$ in $\R^n$. For every $\sigma\in \Delta$, let $\mathbf c_\sigma\in \N^{n}$ be the vector of powers of $t_{i,j}$ in the parameterization of $x_\sigma$. The polytope associated to the affine toric variety $T_{\Delta}$ is $$P\;=\;{\rm conv}\{\mathbf c_\sigma\in \N^n:\,\sigma\in \Delta,\dim(\sigma)>0\}.$$ A cone over $P$, denoted by $C_P$, is the cone generated by vectors in $\{1\}\times P\subset \mathbb R^{n+1}$. The extra coordinate is denoted by $x_0$. Note that $P$ is normal, if and only if the variety $\widehat T_{\Delta}$ is normal, if and only if the (projective) secant variety is normal. Therefore, first we want to prove that $P$ is a normal polytope. To this end, let $L_P\subset \R^{n+1}$ be the lattice generated by $\{1\}\times P$, i.e.~those points that are integral combinations of lattice points of $\{1\}\times P$. Always $L_p\subseteq\Z^{n+1}$ and often equality holds. Recall that a polytope $P$ is called \emph{normal} if for every point $q\in L_P\cap C_P$ we have $q=\sum_{i=1}^s p_i$ for some $p_i\in (\{1\}\times P)\cap\Z^{n+1}$. Note that we must have $q_0=s$, where $q_0$ is the first coordinate of $q$, and we may identify $q$ with a point in $sP$. 



\begin{lema}\label{lem:Pnormal}
	The polytope $P$ is a normal polytope and $C_P$ is defined by the following set of inequalities:
		\begin{itemize}
		\item [(1)] $x_{i,j}\geq 0$ for all $i=1,\ldots,k$, $j=1,\ldots,b_i$,
		\item [(2)] $\sum_{j} x_{i,j}\leq a_ix_0$ for all $i=1,\ldots,k$, and
		\item [(3)] $\sum_{i,j} x_{i,j}\geq 2x_0$.
	\end{itemize}
\end{lema}
\begin{proof}
Let $\bar{C}$ denote the cone defined by the above inequalities $(1)-(3)$. We prove that $\bar{C}=C_P$. It is easily seen that $\bar{C}\supseteq C_P$. The reverse inclusion will be proved together with normality of $P$. 
We prove, by induction on $m$, that any integral point $q\in \bar{C}$ with $q_0=m$ is a sum of $m$ integral points of $\{1\}\times P$. 
The case $m=1$ is clear by definition of $P$. 
Now, assume that $m>1$ and $e_0=(1,0,\dots,0) \in \Z \times \Z^n$. If $q'=q-e_0$ satisfies inequalities of type (2), then $\sum_j q_{i,j}\leq (m-1)a_i$ for all $i=1,\ldots,k$. Hence, by using the fact that $q$ satisfies (3), we may write $q=q^1+q^2$, where $q^2_0=1$ and $q^2$ has exactly two other coordinates equal to $1$ or one coordinate equal to $2$, and all others are zero in both cases. This implies, by induction, that $q^1$ is a sum of $(m-1)$ lattice points in $\{1\}\times P$ and $q^2\in\{1\}\times P$.

If $q'=q-e_0$ does not satisfy $l\geq 2$ inequalities of type (2), then for these $l$ indices $i$ it holds that $(m-1)a_i<\sum_{j}x_{i,j}\leq m a_i$. In consequence, we may write $q=q^1+q^2$, where $q^1$ satisfies the given inequalities, i.e.~by induction it is a sum of $(m-1)$ lattice points of $\{1\}\times P$ and $q^2\in \{1\}\times P$, where $q^2$ has nonzero entries only in coordinates $x_{i,j}$ for $i$ indexing the inequalities not satisfied by $q'$. 

Finally, if $q'$ does not satisfy exactly one inequality of type (2), say indexed by $i_0$, but also $a_{i_0}>1$, then we may conclude as in the previous case. If $a_{i_0}=1$, then we may find $i_1\neq i_0$ such that $q_{i_1,j_1}\geq 1$ for some $j_1$ as $q$ satisfies inequality (3). We may then write $q=q^1+q^2$, where $q^2_0=1$, another coordinate of $q^2$ is equal to one only in one of the variables $x_{i_0,j_0}$ (for some $j_0$) and third nonzero coordinate in variable $x_{i_1,j_1}$. This completes the proof.
\end{proof}

Lemma~\ref{lem:Pnormal} together with Remark~\ref{rem:covered} give us the following proposition.
\begin{obs}\label{prop:Pnormal}
	The variety $\sigma_2(X)$ is covered by toric varieties isomorphic to a product of an affine space of dimension $n=\sum_{i=1}^k b_i$ and the normal toric variety $\widehat T_\Delta$.  \end{obs}

In the remainder of this paper we obtain more information about $\widehat T_\Delta$. For this we need the description of the facets of $P$. Denote the coordinates of the ambient space by $x=(x_{i,j})\in \R^n$ and write:
\begin{itemize}
	\item [(a)] $F=\{x\in P:\;\sum_{i,j} x_{i,j}=2\}$,
	\item [(b)] $R_i=\{x\in P:\sum_{j=1}^{b_i} x_{i,j}=a_i\}$ for $1\leq i\leq k$, and
	\item [(c)] $Z_{i,j}=\{x\in P: x_{i,j}=0\}$ for $1\leq i\leq k$, $1\leq j\leq b_i$.
\end{itemize}
Recall that $n=b_1+\cdots +b_k$. Let $\mathcal I=\{(i,j):1\leq i\leq k, 1\leq j\leq b_i\}$ and note that $|\mathcal I|=n$. 
The canonical unit vectors of $\mathbb R^n$ are denoted by $e_{i,j}$ where $(i,j)\in \mathcal I$. We follow the convention that the elements of $\mathcal I$ are ordered lexicographically. Moreover, from this point onward, without loss of generality, we assume that $a_1\leq a_2\leq \ldots \leq a_k$.

\begin{obs}\label{pSVF}
We have $\dim P=n$, i.e.~the secant variety is of expected dimension $2n+1$, unless: 
\begin{enumerate}
	\item [(D1)] $k=2$, $\mathbf a=(1,1)$,
	\item [(D2)] $k=1$, $a\leq 2$
\end{enumerate}
in which case $P=F$.
In the full dimensional case all the sets $F, R_i$ and $Z_{i,j}$ define facets of $P$ unless:
\begin{enumerate}
	\item [(E1)] $k=3$, $b_{i}=1$ for some $i$ and $a_j=1$ for all $j\neq i$, when $Z_{i,1}$  is not a facet,
	\item [(E2)] $k=2$, $b_1=1$, $a_1\leq 2$, $a_{2}= 2$, when $Z_{1,1}$ is not a facet,
	\item [(E3)] $k=2$, $b_2=1$, $a_1\leq 2$, $a_{2}\geq 2$, when $Z_{2,1}$ is not a facet, and
	\item [(E4)] $k=1$, $b=1$, $a\geq 3$, when $Z_{1,1}$ is not a facet.
\end{enumerate}
\end{obs}

\dow We separately consider cases $k\geq 3$, $k=2$, and $k=1$.
\medskip 

{Case I. $k\geq 3$}

Consider a set of $n$ linearly independent vectors 
\begin{enumerate}
	\item[(1)]  $e_{1,1}+e_{i,j}$ for all $(i,j)\in \mathcal I$ with $i\neq 1$,
	\item[(2)] $e_{1,j}+e_{2,1}$ for all $2\leq j\leq b_1$, and
	\item[(3)] $e_{2,1}+e_{3,1}$
\end{enumerate}
that all lie in $F$. In addition $e_{1,1}+e_{2,1}+e_{3,1}\in P\setminus F$, which shows that $P\subset \mathbb R^n$ is full-dimensional and $F$ is a facet. To show that $R_1$ also forms a facet when $k\geq 3$ (the same proof applies to every $R_i$) note that the following $n$ vectors
\begin{itemize}
	\item [(1)] $a_1e_{1,j}+e_{2,1}$ for all $1\leq j\leq b_1$,
	\item [(2)] $a_1e_{1,1}+e_{i,j}$ for all $(i,j)\in \mathcal I\setminus \{ (2,1)\}$, $i>1$, and
	\item [(3)] $a_1e_{1,1}+e_{2,1}+e_{3,1}$
\end{itemize}
lie in $R_1$ and are linearly independent. Finally, to study the set $Z_{1,1}$ (the same applies to every $Z_{i,j}$) consider the following set of $n-2$ vectors
\begin{itemize}
	\item [(1)] $e_{2,1}+e_{i,j}$ for all $(i,j)\in \mathcal I\setminus \{(1,1)\}$, $i\neq 2$, and
	\item [(2)] $e_{2,j}+e_{3,1}$ for all $2\leq j\leq b_2$.
\end{itemize}
that are linearly independent and lie in $Z_{1,1}$. If $k\geq 4$ we can add to this set $e_{2,1}+e_{3,1}+e_{4,1}$ and $e_{3,1}+e_{4,1}$ to get a set of $n$ affinely independent points in $Z_{1,1}$ and conclude that $Z_{1,1}$ forms a facet in this case. If $k=3$ and $b_1\geq 2$ then we can add $e_{1,2}+e_{3,1}$ and $e_{1,2}+e_{2,1}+e_{3,1}$ to get a set of $n$ affinely independent points in $Z_{1,1}$. If $k=3$, $b_1=1$ 
and $a_2\geq 2$ (the same applies if $a_3\geq2$), then we add $2e_{2,1}$ and $2e_{2,1}+e_{3,1}$ to get $n$ affinely independent points in $Z_{1,1}$. 
So the only exceptional case is when $k=3$, $b_1 =a_2=a_3=1$. In this case, the only integer points in $Z_{1,1}$ are of the form $e_{2,j}+e_{3,j'}$. They are all contained in $F$ showing (E1). 
\medskip

{Case II. $k=2$}

Consider the following set of $n-1$ linearly independent points in $P$:
\begin{enumerate}
	\item[(1)] $e_{1,1}+e_{2,j}$ for all $1\leq j\leq b_2$, and
	\item[(2)] $e_{1,j}+e_{2,1}$ for all $2\leq j\leq b_1$.
\end{enumerate}
If $\mathbf a=(1,1)$ there is no larger affinely independent subset in $P$ since $P$ is contained in hyperplanes $\sum_{j}x_{1,j}=\sum_j x_{2,j}=1$ and so $\dim P=n-2$ proving (D1). In this case $P=F=R_1=R_2$.

\underline{Hence, suppose $a_2\geq 2$.} 

In this case we can add $2e_{2,1}$ and $e_{1,1}+2e_{2,1}$ to get a set of $n+1$ affinely independent points in $P$ and so $\dim P=n$. All these points apart from the last one lie in $F$ so it forms a facet of $P$. The set $R_2$ forms a facet because it contains $n$ affinely independent points: $a_2 e_{2,j}$ for all $1\leq j\leq b_2$ and $e_{1,j}+a_2e_{2,1}$ for all $1\leq j\leq b_1$. Similarly, $R_1$ forms a facet because it contains $a_1e_{1,j}+e_{2,1}$ for all $1\leq j\leq b_1$ and $a_1e_{1,1}+e_{2,j}$ for all $2\leq j\leq b_2$ together with $a_1e_{1,1}+2e_{2,1}$. It remains to check when $Z_{i,j}$ form facets of $P$. 

We first show that $Z_{i,j}$ fails to be a facet of $P$ only if $b_i=1$ (in particular $i=j=1$). If $b_1\geq 2$ then $e_{1,2}+e_{2,j}$ for all $1\leq j\leq b_2$ and $e_{1,j}+e_{2,1}$ for all $3\leq j\leq b_1$ together with $2e_{2,1}$ and $e_{1,2}+2e_{2,1}$ give a set of $n$ affinely independent points in $Z_{1,1}$, proving it forms a facet of $P$. The same argument works for each $Z_{1,j}$. If $b_2\geq 2$ then $2e_{2,j}$ for all $2\leq j\leq b_2$ and $e_{1,j}+e_{2,2}$ for all $1\leq j\leq b_1$ together with $e_{1,1}+2e_{2,2}$ give a set of $n$ affinely independent points in $Z_{2,1}$, proving it forms a facet of $P$. The same argument works for each $Z_{2,j}$. This shows that, given $a_2\geq 2$, $b_{i_0}=1$ is a necessary condition for $Z_{i_0,1}$ to fail to be a facet of $P$. We will consider three cases: (1) $b_1=1,b_2\geq 2$, (2) $b_1\geq 2,b_2=1$, and (3) $b_1=b_2=1$. If $b_1=1$ and $b_2\geq 2$ we have two subcases: (i) $a_2=2$, and (ii) $a_2\geq 3$. If $a_2=2$ then $Z_{1,1}$ is contained in hyperplanes $x_{1,1}=\sum_j x_{2,j}-2=0$ and so it does not form a facet of $P$ ($P$ is full-dimensional), which gives part of (E2), call it (E2.1). If $a_2\geq 3$ then $Z_{1,1}$ forms a facet of $P$ because $2e_{2,j}$ for all $1\leq j\leq b_2$ together with $2e_{2,1}+e_{2,2}$ form a set of $n$ affinely independent points in $Z_{1,1}$. Now, if $b_1\geq 2$ and $b_2=1$ then $Z_{2,1}$ is empty if $a_1=1$ and it is contained in hyperplanes $x_{2,1}=\sum_j x_{1,j}-2=0$ if $a_1=2$, which gives part of (E3), call it (E3.1). If $a_1\geq 3$ we proceed as above to show that $Z_{2,1}$ is a facet of $P$. Finally, suppose $b_1=b_2=1$. The only integral points in $Z_{1,1}$ are $(0,2),\ldots,(0,a_2)$, so $Z_{1,1}$ forms a facet if $a_2\geq 3$ and it is a point if $a_2=2$ giving a part of (E2), call it (E2.2). Note that (E2.1) and (E2.2) already give (E2). On the other hand $Z_{2,1}$ is empty if $a_1=1$ and otherwise it contains the following integral points $(2,0),\ldots,(a_1,0)$, so it forms a facet of $P$ if $a_1\geq 3$ and it is a point if $a_1=2$ showing a part of (E3), call it (E3.2). Note that (E3.1) and (E3.2) already give (E3).
\medskip

{Case III. $k=1$}

If $a_1=1$ then $P$ is empty, showing part of (D2), which we refer to later as (D2.1). Suppose $a_1\geq 2$. If $b_1=1$ then the only integer points in $P$ are $2,\ldots,a_1$. If $a_1=2$ then $P$ is a point, showing another part of (D2), call it (D2.2). If $a_1\geq 3$ then $P$ is full dimensional and $F$, $R_1$ are the two facets of $P$. The set $Z_{1,1}$ is empty showing (E4). If $b_1\geq 2$ and $a_1=2$ then $\dim P=n-1$ since $P$ is contained in hyperplane $\sum_j x_{1,j}=2$ and it contains $n$ linearly independent vectors $2e_{1,j}$ for all $1\leq j\leq b_1$. This shows a part of (D2), which we call (D2.3). Note that (D2.1)-(D2.3) already give (D2). Finally, if $b_1\geq 2$ and $a_1\geq 3$ then the points $2e_{1,j}$ for all $1\leq j\leq b_1$ together with $2e_{1,1}+e_{1,2}$ are $n+1$ affinely independent points in $P$, showing $\dim P=n$. The first $n$ of these points lie in $F$, showing it forms a facet. The points $a_1e_{1,j}$ for all $1\leq j\leq b_1$ lie in $R_1$, showing it is also a facet. Also each $Z_{1,j_0}$ forms a facet of the polytope $P$ as it contains $n$ affinely independent points $2e_{1,j}$ for all $j\neq j_0$ and $3e_{1,j}$ for some $j\neq j_0$. This shows (E4).
\kdow

\begin{figure}[htp!]
\begin{minipage}{0.3\textwidth}
		\begin{tikzpicture}[scale=.6]
\begin{axis}[view={105}{30}]
\addplot3[fill=blue,opacity=0.5] coordinates{(1,1,0) (1,2,0) (1,0,2) (1,0,1) (1,1,0)};
\addplot3[fill=red,opacity=0.5] coordinates {(1,1,0) (1,2,0) (0,2,0) (1,1,0)};
\addplot3[fill=orange,opacity=0.5] coordinates {(1,2,0) (0,2,0) (0,0,2) (1,0,2) (1,2,0)};
\addplot3[fill=yellow,opacity=0.5] coordinates {(0,2,0) (0,0,2) (1,0,1) (1,1,0) (0,2,0)};
\addplot3[fill=green,opacity=0.5] coordinates {(0,0,2) (1,0,2) (1,0,1) (0,0,2)};
\end{axis}
\end{tikzpicture}
\begin{center}
	(1)
\end{center}
\end{minipage}
\begin{minipage}{0.3\textwidth}
	\begin{tikzpicture}[scale=.6]
\begin{axis}[view={105}{30}]
\addplot3[fill=blue,opacity=0.5] coordinates{(1,0,2) (1,0,1) (0,1,1) (0,1,2) (1,0,2)};
\addplot3[fill=red,opacity=0.5] coordinates {(0,1,2) (0,0,2) (0,1,1) (0,1,2)};
\addplot3[fill=orange,opacity=0.5] coordinates {(0,1,1) (0,0,2) (1,0,1) (0,1,1)};
\addplot3[fill=yellow,opacity=0.5] coordinates {(1,0,1) (0,0,2) (1,0,2) (1,0,1)};
\addplot3[fill=green,opacity=0.5] coordinates {(1,0,2) (0,1,2) (0,0,2) (1,0,2)};
\end{axis}
\end{tikzpicture}
\begin{center}
	(2)
\end{center}
\end{minipage}
\begin{minipage}{0.3\textwidth}
	\begin{tikzpicture}[scale=.6]
	\begin{axis}[view={105}{30}]
\addplot3[fill=brown,opacity=0.5]coordinates{(0,2,0) (1,2,0) (1,1,0) (0,2,0)};
\addplot3[fill=blue,opacity=0.5] coordinates{(0,0,3) (0,2,3) (0,2,0) (0,0,2) (0,0,3)};
\addplot3[fill=red,opacity=0.5] coordinates {(0,0,3) (0,2,3) (1,2,3) (1,0,3) (0,0,3)};
\addplot3[fill=orange,opacity=0.5] coordinates {(0,2,3) (0,2,0) (1,2,0) (1,2,3) (0,2,3)};
\addplot3[fill=yellow,opacity=0.5] coordinates {(0,2,0) (1,1,0) (1,0,1) (0,0,2) (0,2,0)};
\addplot3[fill=green,opacity=0.5] coordinates {(1,0,1) (1,0,3) (0,0,3) (0,0,2) (1,0,1)};
\end{axis}
\end{tikzpicture}
\begin{center}
	(3)
\end{center}
\end{minipage}
\caption{The polytopes of Example~\ref{ex:polytopes}.}\label{fig:polytopes}
\end{figure}
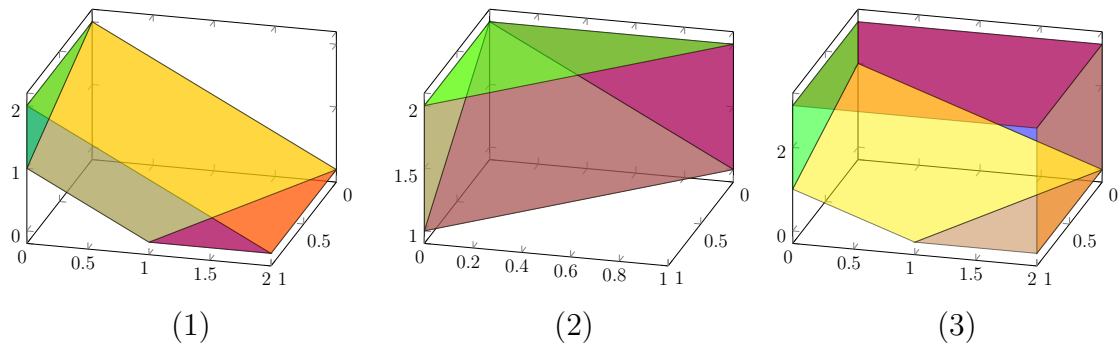

\ex\label{ex:polytopes}The following three polytopes are given in Figure~\ref{fig:polytopes}:\\
$(1)$ Let $k=2,\mathbf a=(1,2),\mathbf b=(1,2)$ and hence the inequalities defining $P$ in the nonnegative orthant are
\[
x_{1,1}\leq1, \quad x_{2,1}+x_{2,2}\leq2, \quad\text{ and } \quad x_{1,1}+x_{2,1}+x_{2,2}\geq2.
\]
$(2)$ Let $k=2,\mathbf a=(1,2),\mathbf b=(2,1)$, hence the inequalities defining $P$ in the nonnegative orthant are
\[
x_{1,1}+x_{1,2}\leq1, \quad x_{2,1}\leq2, \quad \text{ and } \quad x_{1,1}+x_{1,2}+x_{2,1}\geq2.
\]
Note that, as predicted by case (E3) in Proposition~\ref{pSVF}, $Z_{2,1}$ does not form a facet of $P$ and so $P$ has five facets. \\
$(3)$ Let $k=3,\mathbf a=(1,2,3),\mathbf b=(1,1,1)$, hence the inequalities defining $P$ in the nonnegative orthant are
\[
x_{1,1}\leq 1, \quad x_{2,1}\leq2,\quad x_{3,1}\leq 3,\quad \text{ and } \quad x_{1,1}+x_{2,1}+x_{3	,1}\geq2.
\]
\kex
The conclusion about the dimension in Proposition \ref{pSVF} was known and can be proved using the Teraccini Lemma. However, the detailed description of facets allows us to study the geometry of the secant variety in much greater detail, in particular the nature of singularities, which we do in the next sections.

\section{Singularities of the secant variety}\label{sec:singularity}

In this section we apply Proposition~\ref{pSVF} to study singularities of the secant variety of the Segre-Veronese variety $X$.  Note that for a variety $Y$, the secant variety $\sigma_2(Y)$ is smooth if and only if it fills the whole ambient space. The results when the secant of the Segre-Veronese variety $X$ fills the ambient space are quite standard; c.f. Theorem~\ref{GSV} below. Thus the first interesting question concerns classification of the cases when singularities are Gorenstein.


By Proposition~\ref{prop:Pnormal}, $\sigma_2(X)$ can be covered by normal affine toric varieties isomorphic to a product of an affine space and $\widehat T_\Delta$. To study Gorenstein varieties we can employ \cite[Proposition~8.2.12]{cox2011toric}, which we include for reader's convenience adapted to $\widehat T_\Delta$.
\begin{obs}\label{prop:gorencond}
If the polytope $P$ of the toric variety $\widehat T_\Delta$ is full-dimensional then $\widehat T_\Delta$ is smooth if and only if the ray generators of $C_P$ form a basis of the lattice $\mathbb Z\times L_P$. Moreover, $\widehat T_\Delta$ is Gorenstein if and only if the rays $u$ of the dual cone $C_P^\vee$ belong to an affine subspace $\langle \beta,u\rangle=1$ for some integral point $\beta\in \mathbb Z\times L_P$.
\end{obs}
Geometrically, the condition $\langle\beta,u\rangle=1$ encodes a point $\beta$ in the lattice $\mathbb Z\times L_P$ that lies at distance one from all facets of $C_P$. If $P$ is not full-dimensional we can still look for such a point but taking care of the minimal lattice containing all lattice points of $P$. The condition for smoothness can also be easily extended to the case when $P$ is not full-dimensional. We illustrate this with two examples. 
\begin{exm}\label{ex:smooth}
	Suppose $k=2$, $\mathbf a=(1,1)$ and $\mathbf b=(1,2)$. The integral points in $P$ are  $(1,0,1)$, $(1,1,0)$. In this case $P=F=R_1=R_2$ and $\dim P=1$. The two ray generators of $C_P$ are $(1,1,0,1)$ and $(1,1,1,0)$. All the lattice points in $C_P\subset \R\times \R^3$ are of the form $(n_1+n_2,n_1+n_2,n_2,n_1)$ for some $n_1,n_2\in \N$ and so the ray generators are a basis of $\mathbb Z\times L_P$. This shows that the associated secant variety is smooth. 
\end{exm}

\begin{exm}\label{ex:goren}
	Suppose $k=2$, $\mathbf a=(1,1)$ and $\mathbf b=(2,2)$. Then $P$ is the product of 1-simplices with four integral points $(0,1,0,1)$, $(0,1,1,0)$, $(1,0,0,1)$, $(1,0,1,0)$. In this case again $P=F=R_1=R_2$ and $\dim P=2$. The sets $Z_{1,1}$, $Z_{1,2}$, $Z_{2,1}$, $Z_{2,2}$ all form facets of $P$. The point $(2,1,1,1,1)\in  C_P$ lies (in the lattice) at distance one from all the four facets of $C_P$, which confirms that the corresponding secant variety is Gorenstein. It is however \emph{not} smooth. 	\end{exm}

We are ready to present the complete classification of the Gorenstein property of the secant of the Segre-Veronese variety. We keep our convention of ordering elements of $\mathcal I$ lexicographically. 
\begin{thm}\label{GSV}
	The secant variety of the Segre-Veronese variety is smooth if and only if
	\begin{itemize}
		\item [(S1)] $\mathbf a=(1,1,1)$, $\mathbf b=(1,1,1)$,
		\item [(S2)] $\mathbf a=(1,1)$, $b_1=1$, or
		\item [(S3)] $a=1$ or ($a=2$ and $b=1$).
	\end{itemize}
	If the secant variety of the Segre-Veronese variety is not smooth, then it is Gorenstein if and only if one of the following holds
	\begin{itemize}
		\item [(G1)] $\mathbf a=(1,1,1,1,1)$, $\mathbf b=(1,1,1,1,1)$,
		\item [(G2)] $\mathbf a=(1,1,1)$, $\mathbf b$ equals one of $(1,1,3),(1,3,3)$ or $(3,3,3)$,
		\item [(G3)] $\mathbf a=(1,1,2)$, $\mathbf b$ equals $(1,1,1)$ or $(1,1,3)$,
		\item [(G4)] $\mathbf a=(2,2,2)$, $\mathbf b=(1,1,1)$,
		\item [(G5)] $\mathbf a=(1,1)$, $b_1=b_2, b_1>1$,
		\item [(G6)] $\mathbf a=(1,2)$, $\mathbf b$ equals one of $(1,1), (1,5), (2,1)$ or $(2,5)$,
		\item [(G7)] $\mathbf a=(2,3)$, $\mathbf b$ equals $(1,1)$ or $(1,2)$,
		\item [(G8)] $a=2$, $b$ is even,
		\item [(G9)] $a=3$, $b$ equals $1$ or $5$,
		\item [(G10)] $a=4$, $b$ equals $1$ or $3$.
	\end{itemize}
\end{thm}

\dow

For fixed data $\mathbf a,\mathbf b,k$ define the following vectors in $\mathbb Z\times \mathbb Z^n$ 
$$
u_F=-2e_0+\sum_{i=1}^k\sum_{j=1}^{b_i}e_{i,j}\qquad\mbox{and}\qquad u_{R_i}=a_ie_0-\sum_{j=1}^{b_i}e_{i,j}\qquad\mbox{and}\qquad u_{Z_{i,j}}=e_{i,j}.
$$
Let $\mathcal U=\{u_F,u_{R_1},\ldots,u_{R_{k}},u_{Z_{1,1}},\ldots,u_{Z_{k,b_k}}\}$. These vectors correspond to the faces $F, R_i, Z_{i,j}$. We use the facet description given in Proposition~\ref{pSVF}. By Proposition~\ref{prop:gorencond}, $\widehat T_\Delta$ is Gorenstein if and only if ray generators $u$ of $C_P^\vee$ belong  to an affine subspace $\langle \beta,u\rangle=1$ for some $\beta=(\beta_0, \beta_{i,j})\in \mathbb Z\times L_P$. 

We separately consider the cases when $P$ is full-dimensional and when it has positive codimension. In the full-dimensional case we separately consider cases depending on which vectors in $\mathcal U$ define facets of $P$. We start with the simplest, but most general case.

\medskip
Case I.1: $\dim P=n$ and all vectors in $\mathcal U$ define facets

If all the faces $F,R_i, Z_{i,j}$ are facets, then the rays of $C_P^\vee$ are precisely the vectors in $\mathcal U$. By evaluating on vectors $u_{Z_{i,j}}$, we get that $\beta_{i,j}=1$ for all $(i,j)\in \mathcal I$. By evaluating on the remaining elements of $\mathcal U$, we obtain that if $\beta$ as above exists, it must in addition satisfy
\begin{equation}\label{eq:condbeta}
-2\beta_0+\sum_{i=1}^k b_i=1 \qquad\mbox{and}\qquad a_i\beta_0-b_i=1 \quad\mbox{for all } 1\leq i\leq k.
\end{equation}
%

We can solve the second set of equations in (\ref{eq:condbeta}) for $b_i$ and plug into the first equation to get $\beta_0\left((\sum a_i)-2\right)=k+1$. If $k=1,2$ then, by Proposition~\ref{pSVF}, $\sum a_i\geq 3$ and for all $k$ it trivially holds that $\sum a_i\geq k$. Thus, it is clear that (\ref{eq:condbeta}) has \emph{no} solution if $\beta_0>4$. If $\beta_0=4$ then (\ref{eq:condbeta}) becomes $\sum_i b_i=9$ and $4a_i=b_i+1$ for $1\leq i\leq k$. This means that $b_i\in \{3,7,11,\ldots\}$ and $b_i$'s sum to $9$. The only way this is possible, is when $k=3$, $\mathbf b=(3,3,3)$, and $\mathbf a=(1,1,1)$ providing a part of $(G2)$. If $\beta_0=3$ then necessarily $b_i\in \{2,5,8,11,\ldots\}$ and $b_i$'s sum to $7$. This is only possible if $\mathbf a=(1,2)$, $\mathbf b=(2,5)$, providing part of (G6). If $\beta_0=2$, (\ref{eq:condbeta}) implies that $b_i\in \{1,3,5,\ldots\}$ and $\sum_i b_i=5$. There are three possible solutions: (i) $k=1$, $a=3$, $b=5$ (part of (G9)), (ii) $k=3$, $\mathbf a=(1,1,2)$, $\mathbf b=(1,1,3)$ (part of (G3)), (iii) $k=5$, $\mathbf a=(1,1,1,1,1)$, $\mathbf b=(1,1,1,1,1)$ (gives (G1)). If $\beta_0=1$, then $\sum_i b_i=3$. There are three possible solutions: (i) $k=1$, $a=4$, $b=3$ (part of (G10)), (ii) $k=2$, $\mathbf a=(2,3)$, $\mathbf b=(1,2)$ (part of (G7)), (iii) $k=3$, $\mathbf a=(2,2,2)$, $\mathbf b=(1,1,1)$ (gives (G4)).

We now discuss the exceptional cases of Proposition~\ref{pSVF} when some $Z_{i,j}$ do not form facets (necessarily $k\leq 3$) but $P$ is still full-dimensional. 

\medskip
Case I.2: $\dim P=n$ and none of $Z_{i,j}$ form a facet

In this case the rays of $C_P^\vee$ are $u_F$ and $u_{R_i}$ for $1\leq i\leq k$. If $\beta$ exists such that $\langle \beta,u\rangle=1$ for all rays $u$, it must satisfy
\begin{equation}\label{eq:condbeta2}
-2\beta_0+\sum_{i=1}^k\sum_{j=1}^{b_i} \beta_{i,j}=1 \qquad\mbox{and}\qquad a_i\beta_0-\sum_{j=1}^{b_i}\beta_{i,j}=1 \quad\mbox{for all } 1\leq i\leq k.
\end{equation}
In particular, $(\sum_i a_i-2)\beta_0=k+1$. If $k=1$, by Proposition~\ref{pSVF}, we have  $b=1$ and $a\geq 3$. By the above equations, we have two options to get an integer solution: (i) $a=3$ (then $\beta_0=2$, $\beta_{1,1}=5$) giving part of (G9), and (ii) $a=4$ (then $\beta_0=1$, $\beta_{1,1}=3$) which gives part of (G10). We note in passing that if $a\geq 5$ then $\beta_0=2/(a-2)$, $\beta_{1,1}=(a+2)/(a-2)$ is always a valid rational solution, which confirms (Q4) in Theorem~\ref{QGSV}.

If  $k=2$, by Proposition~\ref{pSVF}, we have $b_1=b_2=1$ and either $\mathbf a=(1,2)$ or $\mathbf a=(2,2)$.  In the first case, by (\ref{eq:condbeta2}), we get $\beta_0=3$, $\beta_{1,1}=2$, $\beta_{2,1}=5$, which gives part of (G6). In the second case, $\beta_0=3/2$, $\beta_{1,1}=\beta_{2,1}=2$ and so there is no integer solution. We note in passing that this gives part of (Q1) in Theorem~\ref{QGSV}. If $k=3$, by Proposition~\ref{pSVF}, we have have $\mathbf a=(1,1,1)$ and $\mathbf b=(1,1,1)$, in which case the rays of the dual cone $C_P^\vee$ form a basis of the lattice  $\Z\times L_P$  and so we obtain the smooth case $(S1)$. We cannot have $k\geq 4$ by Proposition~\ref{pSVF}.

\medskip
Case I.3: $\dim P=n$ and exactly two $Z_{i,j}$ fail to form a facet

If $k=1$ this is not possible by Proposition~\ref{pSVF}. If $k=2$, again by Proposition~\ref{pSVF}, this is only possible when $b_1=b_2=1$ and so when \emph{none} of $Z_{i,j}$ are facets. This was already covered in Case I.2. Finally, if $k=3$, we necessarily have that $\mathbf a=(1,1,1)$ and $\mathbf b$ has exactly two entries equal to one. Up to symmetry $\mathbf b=(1,1,b_3)$. Since $Z_{3,j}$ for $1\leq j\leq b_3$ all form facets, equation $\langle \beta,u\rangle=1 $ for all rays $u$ of $C_P^\vee$ implies that $\beta_{3,j}=1$ for all $1\leq j\leq b_3$. Now evaluating this form on $u_{R_3}=e_0-\sum_{j}e_{3,j}$ we get that $\beta_0-b_3=1$. Using $u_F$, $u_{R_1}$, and $u_{R_2}$ we conclude that $\beta_0=4$ and $\beta_{1,1}=\beta_{2,1}=b_3=3$. This gives a part of $(G2)$. 

\medskip
Case I.4: $\dim P=n$ and exactly one $Z_{i,j}$ fails to form a facet

If $k=1$ this is only possible when $b=1$, which was covered in Case I.2. If $k=2$, by Proposition~\ref{pSVF}, we have three possible cases
\begin{enumerate}
	\item [(i)] $b_1=1$, $b_2\geq 2$, $a_1\leq 2$, $a_2=2$, when only $Z_{1,1}$ fails to form a facet,
	\item [(ii)] $b_2=1$, $a_1\leq 2$, $a_2\geq 3$, when only $Z_{2,1}$ fails to form a facet,
	\item [(iii)] $b_1\geq 2$, $b_2=1$, $a_1\leq 2$, $a_2= 2$, when only $Z_{2,1}$ fails to form a facet.
\end{enumerate}
In case (i), evaluating $\langle\beta,u\rangle =1$ on the rays $u$ of $C_P^\vee$, we get that $\beta_{2,j}=1$ for all $1\leq j \leq b_2$ and
$$
-2\beta_0+\beta_{1,1}+b_2=1,\;\;a_1\beta_0-\beta_{1,1}=1,\;\;2\beta_0-b_2=1.
$$ 
In particular, $\beta_{1,1}=2$, which reduces this system to two equations: $2\beta_0-b_2=1$, $a_1\beta_0=3$. If $a_1=1$ then $\beta_0=3$ and so $\mathbf a=(1,2)$, $\mathbf b=(1,5)$, which confirms part of (G6). If $a_1=2$ then $\beta_0=3/2$ and so the solution is not integer. We note in passing that the solution is rational and in that case $\mathbf a=(2,2)$ and $\mathbf b=(1,2)$, which confirms part of (Q1) in Theorem~\ref{QGSV}.

In case (ii), we get $\beta_{1,j}=1$ for all $1\leq j\leq b_1$ and  
$$
-2\beta_0+b_1+\beta_{2,1}=1,\;\;a_1\beta_0-b_1=1,\;\;a_2\beta_0-\beta_{2,1}=1.
$$ 
This in particular implies that $(a_1+a_2-2)\beta_0=3$, which has integer solutions only if either $a_1+a_2=5$, $\beta_0=1$, or if $a_1+a_2=3$, $\beta_0=3$. Since $a_2\geq 3$ the only possibilities are $\mathbf a=(1,4)$ and $\mathbf a=(2,3)$. The former leads to no valid solution and if $\mathbf a=(2,3)$ we get $\mathbf b=(1,1)$, $\beta_{2,1}=2$, which confirms part of (G7). For the proof of Theorem~\ref{QGSV} we also want to consider when the above system admits rational solutions. In case (ii) we either have $a_1=1$ or $a_1=2$. If $a_1=1$ then $\beta_0=3/(a_1+a_2-2)\leq 3/2$ and equation $a_1\beta_0-b_1=1$ implies that $b_1\leq 1/2$, which is impossible. If $a_1=2$ then by a similar argument necessarily $b_1=1$, $a_2=3$, which was already covered above. We conclude there are no extra rational solutions possible. 

In case (iii), we get $\beta_{1,j}=1$ for all $1\leq j\leq b_1$ and  
$$
-2\beta_0+b_1+\beta_{2,1}=1,\;\;a_1\beta_0-b_1=1,\;\;2\beta_0-\beta_{2,1}=1.
$$ 
This in particular implies that $a_1\beta_0=3$, which has integer solutions only if either $a_1=3$ (impossible), or if $a_1=1$, $\beta_0=3$. This gives that $\mathbf a=(1,2)$, $\mathbf b=(2,1)$, $\beta_0=3$, $\beta_{2,1}=5$ giving part of (G6). If $a_1=2$ we get an additional rational solution with $\beta_0=3/2$. In this case $\mathbf a=(2,2)$, $\mathbf b=(2,1)$, which confirms part of (Q1) in Theorem~\ref{QGSV}. Note that the varieties for $\mathbf a=(2,2)$, $\mathbf b=(2,1)$ and $\mathbf a=(2,2)$, $\mathbf b=(1,2)$ are isomorphic.

 If $k=3$, we necessarily have $\mathbf a=(1,1,a_3)$ and $\mathbf b=(b_1,b_2,1)$. Evaluating $\langle \beta,u\rangle=1$ on all $u\in \mathcal U\setminus \{u_{Z_{3,1}}\}$ we get $\beta_{1,i}=\beta_{2,j}=1$ for all $i,j$ and 
$$
-2\beta_0+b_1+b_2+\beta_{3,1}=1,\;\;\beta_0-b_1=1,\;\;\beta_0-b_2=1,\;\;a_3\beta_0-\beta_{3,1}=1.
$$ 
This implies that $a_3\beta_0=4$, which has three integral solutions. If $a_3=1$, $\beta_0=4$ we obtain $\mathbf a=(1,1,1)$, $\mathbf b=(3,3,1)$, which gives part of (G2). If $a_3=2$, $\beta_0=2$ we obtain $\mathbf a=(1,1,2), \mathbf b=(1,1,1)$, part of (G3). If $a_3=4$, $\beta_0=1$ there is no positive solution. Since $b_i=\beta_0-1$, no further rational solutions are possible. 

Finally, Proposition~\ref{pSVF} confirms that a situation when more than two but not all $Z_{i,j}$ fail to be facets is impossible. This concludes the proof of Case I. 

\medskip
Case II: $\dim P<n$

We have two main cases to consider. 

\medskip
Case II.1: $k=2$, $\mathbf a=(1,1)$ (this case was classically known, c.f.~Remark \ref{GSVr})

The only integer points contained in $P$ are of the form $e_{1,j}+e_{2,j'}$ for $1\leq j\leq b_1$, $1\leq j'\leq b_2$; see also Example~\ref{ex:smooth} and Example~\ref{ex:goren}. Thus $P=F=R_1=R_2$ and $\dim P=n-2$ ($P$ is contained in the affine space $\sum_j x_{1,j}=\sum_{j}x_{2,j}=1$). 

Assume first that $b_2\geq b_1\geq 2$ then $e_{1,2}+e_{2,j}$ for $1\leq j\leq b_2$ and $e_{1,j}+e_{2,1}$ for $3\leq j\leq b_1$ (note that there is no such point if $b_1=2$) form a set of $n-2$ affinely independent points in $Z_{1,1}$. In this case $Z_{1,1}$ forms a facet of $P$ and the same argument applies to all other $Z_{i,j}$. If $b_1=b_2$ then the point $\beta=b_1e_0+\sum_{j}e_{1,j}+\sum_{j}e_{2,j}$ lies in $C_P\cap (\mathbb Z\times L_P)$ and is in the lattice at distance one from the facets of $C_P$. This gives (G5). If $b_1<b_2$ there is no such point. 

Suppose now that $b_1=1$. In this case $Z_{1,1}=\emptyset$ and so it does not form a facet of $P$ (unless $b_2=1$ in which case $P$ is the point $(1,1)$). Thus, if $b_1=b_2=1$, the corresponding variety is smooth which gives part of (S2). However, if $b_2\geq 2$ then $Z_{1,1}$ is not a facet but each $Z_{2,j}$ is a facet since it contains $n-2$ affinely independent points $e_{1,1}+e_{2,j'}$ for $j'\neq j$. Since these points $e_{1,1}+e_{2,j'}$, for all $1\leq j'\leq b_2$, of the polytope $P$ are linearly independent and form a basis of $L_P$, we conclude that the corresponding variety is smooth giving part of (S2).

\medskip
Case II.2: $k=1$, $a\leq 2$

If $a=1$ then $P$ is empty giving part of (S3).

If $a=2$ then $P=F=R_1$ and $\dim P=n-1$ as the polytope is contained in the affine space $\sum_j x_{1,j}=2$. If $b\geq 2$ then $Z_{1,j}$ all form facets of $P$, which we confirm for $Z_{1,1}$  by taking points $2e_{1,j}$ for $2\leq j\leq b_1$. The point $\beta=(b/2,1,\ldots,1)$ lies at distance 1 to the facets and so the corresponding variety is Gorenstein if $b$ is even, which gives (G8). If $b=1$ then $P$ is a point, which gives a part of (S3). If $b>1$ is odd then  the corresponding variety is $\Q$-Gorenstein, but not Gorenstein, which confirms (Q3) in Theorem~\ref{QGSV}.
\kdow

\uwa\label{GSVr}
For the case $k=2$, the secant of the Segre product is the locus of $(b_1+1)\times (b_2+1)$ matrices of rank at most two. Hence it is  also a determinental variety. In this setting, (G5) of the above classification is classically known as the case of square matrices. It was, after partial attempts by Eagon \cite{eagon1969examples} and Goto  \cite{goto1974determinantal}, proved by Svanes (see \cite[Theorem~5.5.6]{svanes1974coherent}); see \cite[Chapter 8]{bruns2006determinantal} for more discussion. Also, $(G1), (G2)$ and $(G5)$ were proved in \cite{MOZ}. Moreover, $(G8)$ was also considered by Hibi and Ohsugi in \cite{Hibi}, who gave the description of the polytope. Hibi and Ohsugi also used their description of the polytope, in case of a graph having no loops and no multiple edges, in \cite{ohsugi2006special} to investigate Gorenstein property.\kuwa

We now present the complete classification of $\Q-$Gorenstein property of the secant of the Segre-Veronese variety. We again analyse this using local toricness of $\sigma_2(X)$. Since $\widehat T_\Delta$ is a normal toric variety, from \cite[Proposition~8.2.12]{cox2011toric}, it is $\Q-$Gorenstein if and only if the ray generators of the dual cone $C_P^\vee$ belong to an affine subspace $\langle\beta,u\rangle=1$ for some rational point $\beta\in (\Z\times L_P)\otimes \Q$. 
\begin{thm}\label{QGSV}
		The secant variety of the Segre-Veronese variety is $\Q-$Gorenstein but not Gorenstein if and only if
	\begin{itemize}
		\item [(Q1)] $\mathbf a=(2,2)$, $\mathbf b$ equals one of $(1,1), (1,2)$ or $(2,2)$,
		\item [(Q2)] $\mathbf a=(4,4), \mathbf b=(1,1)$,
		\item [(Q3)] $a=2, b>1$ is odd,
		\item [(Q4)] $a\geq5, b=1$, or
		\item [(Q5)] $a=6, b=2$.
	\end{itemize}
\end{thm}
\dow
We follow the same structure as in the proof of Proposition~\ref{GSV}. 

\medskip
Case I.1: $\dim P=n$ and all vectors in $\mathcal U$ define facets

Line in the proof of Proposition~\ref{GSV}, equations $\langle\beta,u\rangle=1$ for $u\in \mathcal U$ imply that $\beta_{i,j}=1$ for all $(i,j)\in \mathcal I$, and further that (\ref{eq:condbeta}) holds. If there is a (rational) solution then $\beta_0=(n-1)/2$ and $\beta_0$ satisfies $\beta_0\leq 4$, it follows that $n\leq 9$ and $\beta_0=\frac{m}{2}$ for $m\in \{1,2,\ldots,8\}$. The cases when $\beta_0\in \{1,2,3,4\}$ were covered in the proof of Proposition~\ref{GSV}. If $\beta_0=7/2$ we get $\sum b_i=8$ and $b_i\in \{6,13,20,\ldots\}$ and so there is no solution.  If $\beta_0=5/2$ then $\sum b_i=6$ and $b_i\in \{4,9,14,\ldots\}$ and again there is no solution. If $\beta_0=3/2$ then $\sum b_i=4$ and $b_i\in \{2,5,8,\ldots\}$ and so $k=2$, $\mathbf b=(2,2)$, $\mathbf a=(2,2)$ is a solution, which gives part of (Q1). If $\beta_0=1/2$ then $\sum b_i=2$. There are two potential solutions. One is $k=2$, $\mathbf a=(4,4)$, $\mathbf b=(1,1)$, which gives (Q2). The other is  $k=1$, $a=6$, $b=2$, which gives (Q5).   

All the remaining cases were covered in the proof of Proposition~\ref{GSV}.
\kdow
Our methods enabled a complete classification of cases when the secant variety is Gorenstein or $\Q-$Gorenstein, i.e.~every local ring of every point in the projective variety is Gorenstein. There is a different, stronger property, sometimes referred to arithmetically Gorenstein, which asks when the ring is Gorenstein after localizing at the zero point of the affine cone. The problem of classification of arithmetically Gorenstein secant varieties remains open even for Segre products. In principle, one should be able to check each particular example from the list we provide on a computer, however this is not doable in practice. 
\begin{question}
Which secant varieties of Segre-Veronese varieties are arithmetically Gorenstein?
\end{question}

\section{Singular locus of the secant variety}\label{sec:locus}

We are now ready to describe the singular locus of the secant of the Segre-Veronese variety, thus extending the results of \cite{MOZ} for the case of the secant of Segre variety. Our description of the singular locus relies on careful analysis of the vertices of the polytope $P$. This gives us understanding of the associated projective variety $V_P$ (c.f.  \cite[Section 2.3]{cox2011toric} for the details of this construction), which is equal to the projectivization of $T_\Delta$. This can be then directly translated to the singular locus of $\widehat T_\Delta$. This is because a point $p$ of a projective variety $V_P$ is singular if and only if some/any nonzero lift $p'$ of that point is a singular point of the affine cone $\widehat V_P$ over $V_P$. In other words: projectivization of the singular locus is the singular locus of the projectivization.

It is a well known fact that all properties of the polytope $P$ and the associated (projective) variety $V_P$ are encoded in the normal fan $\Sigma$ of $P$, since $P$ is normal by Lemma~\ref{lem:Pnormal}. The maximal cones of $\Sigma$ are constructed as follows (and all the subcones appear as intersection of maximal cones). For each vertex $v$ of $P$, consider the normal vectors to all facets containing $v$ pointing inside the polytope $P$. The cone corresponding to $v$ generated by such normal vectors is denoted by $\sigma_v$. A rational polyhedral cone $\sigma$ is \emph{smooth} if its ray generators form a part of a basis of the lattice. A vertex $v$ is a smooth vertex if and only if $\sigma_v$ is a smooth cone.

%
%


We now classify singular vertices of $P$. Note that if $\Delta_{b_i}=\conv(0,e_{i,1},\ldots,e_{i,b_i})\subset \r^{b_i}$ is the standard $b_i-$simplex, then the polytope associated to the Segre-Veronese variety is $Q:=a_1\Delta_{b_1}\times\cdots\times a_k\Delta_{b_k}$. Hence $v=(v_1,\ldots,v_k)$ is a vertex of $Q$ if and only if  each $v_i$ is a vertex of $a_i\Delta_{b_i}$. Denoting
	\[
	F^+: \sum_{i,j} x_{i,j}\geq 2
	\]
	we can write the polytope $P$ of the secant variety $\sigma_2(X)$ as $P=Q\cap F^+$, see Lemma~\ref{lem:Pnormal}.
	Therefore, a vertex $v$ of $P$ either belongs to $F$ or not, that is, the sum of the coordinates is either equal to two or strictly greater than two. If the sum is strictly greater than two, then a vertex $v$ of $P$ is also a vertex of $Q$, since locally $P$ and $Q$ are the same outside of $F$. In that case $v$ is smooth as all vertices of $Q$ are smooth. Hence, the possible non-smooth vertices of $P$ belong to the hyperplane $F$ and hence are of the  form $e_{i,j}+e_{i',j'}$ or $2e_{i,j}$ (if $a_i\geq 2$). The following Lemma~identifies the vertices of $P$ which lie on the hyperplane $F$. 
\begin{lema}\label{vSLSV} The point $e_{i,j}+e_{i',j'}$ forms a vertex of $P$ if and only if $i\neq i'$ and $\min\{a_{i},a_{i'}\}=1$. The point $2e_{i,j}$ forms a vertex of $P$ if and only if $a_{i}\geq 2$.
\end{lema}	
\begin{proof}
If $\min\{a_{i},a_{i'}\}>1$ then $e_{i,j}+e_{i',j'}=\frac{1}{2}(2e_{i,j})+\frac{1}{2}(2e_{i',j'})$, hence it is not a vertex. If $a_i=1$ then we must have $i\neq i'$. The only lattice points in $P$ with coordinates different from $e_{i,j},e_{i',j'}$ equal to zero are $e_{i,j}+se_{i',j'}$ for $1\leq s\leq a_{i'}$. The point $e_{i,j}+e_{i',j'}$ does not belong to the convex hull of the other points of this type, hence is a vertex of $P$. 

If $a_i=1$ then $2e_{i,j}$ does not belong to $P$. Otherwise, the only lattice points in $P$ with coordinates different from $e_{i,j}$ equal to zero are $se_{i,j}$ for $2\leq s\leq a_i$. The point $2e_{i,j}$ does not belong to the convex hull of the other points of this type, hence is a vertex of $P$. 
\end{proof}

\begin{exm}
	Consider the three situations in Example~\ref{ex:polytopes}. In (1) all vertices are smooth. In (2) the vertex $2e_{2,1}$ has four rays coming out of it and so it is not smooth. In (3) there is only one non-smooth vertex $2e_{2,1}$.
\end{exm}

The following Lemma~classifies the smooth vertices of $P$ which lie on the hyperplane $F$. Note that two vertices $v_1$ and $v_2$ of a polytope $P$ are connected by an edge if there exists a supporting hyperplane of $P$ that contains these two vertices and no other vertex of $P$. Assume that $i_j\not=i_l$ for any $j\not=l$.

\begin{lema}\label{sSLSV}
	The vertex $e_{i_1,j_1}+e_{i_2,j_2}$ ($i_1\neq i_2$) is smooth only in the following cases:
	\begin{itemize}
		\item [(1)] $k=2$;
		\item [(2)] $k=3$ when $a_{i_1}=a_{i_2}=1$ and $b_{i_3}=1$; or
		\item [(3)] $k\geq3$ when $a_{i_1}=1$ and $a_{i_2}\geq2$.
	\end{itemize}
	If $a_i\geq 3$ then the vertex  $2e_{i,j}$ is smooth. If $a_i=2$ it is smooth only in the following cases:
	\begin{itemize}
		\item [(1)] $k=1$; or 
		\item [(2)] $k=2$ when $b_{i'}=1$, where $i'\neq i$.
	\end{itemize}
\end{lema}
\dow
If $k=1$ and $b=1$ then $P$ is either empty, or the interval $[2,a]$ if $a\geq 2$. Hence, it is smooth. If $k=1$ and $b\geq 2$ then we consider two cases: $a=2$ and $a\geq 3$. The points  $e_{1,i}+e_{1,j} ~(i\neq j)$ never form a vertex; c.f. Lemma~\ref{vSLSV}. If $a=2$ then 
the polytope $P$ is twice the standard simplex and hence is smooth. 
If $a\geq 3$ then $\dim P=b$. Each point $2e_{1,i}$ lies in $b$ facets $F$ and $Z_{1,j}$ for $j\neq i$. The normal vectors form a basis of the lattice and hence $P$ is smooth.

If $k\geq 2$ we first consider points of the form $v:=e_{1,1}+e_{2,1}$ (the same argument for any $e_{i_1,j_1}+e_{i_2,j_2}$ applies). If $a_{1},a_2\geq 2$ then, by Lemma~\ref{vSLSV}, $v$ is not a vertex of $P$. Thus, we consider two cases: (a) $a_{1}=a_{2}=1$ and (b) $a_{1}=1, a_{2}\geq 2$.

 In case (a): If $k=2$ then $P=F=R_1=R_2$ and all its vertices are of the form $e_{1,i}+e_{2,j}$. 
The polytope is a product of two standard simplices and so it is smooth. If $k\geq 3$ then $\dim P=n$ by Proposition~\ref{pSVF}. The vertex $v$ lies on $F$, $R_1$, $R_2$, and $Z_{i,j}$ for $(i,j)\in \mathcal I\setminus \{(1,1),(2,1)\}$. Unless $k=3$ and $b_3=1$, this makes $n+1$ facets and so $v$ cannot be smooth. When $k=3$, $b_3=1$, then by Proposition~\ref{pSVF}~(E1), $Z_{3,1}$ does not form a facet. Indeed, then $v$ is smooth. 
This corresponds to the smooth case (2) in the statement of the lemma.
 
  In case (b), $v$ lies on $F$, $R_{1}$, and $Z_{i,j}$ for all $(i,j)\in \mathcal I\setminus\{(1,1),(2,1)\}$. These are precisely $n$ facets, with normals forming a basis of the lattice, and so $v$ is smooth in this case. This gives the smooth case (3).

 Suppose now that $a_l\geq 2$ and consider the vertex $v:=2e_{l,1}$ (with the same argument for any other $2e_{i,j}$). If $a_l\geq 3$, this vertex lies on $F$, and $Z_{i,j}$ for $(i,j)\in \mathcal I\setminus \{(l,1)\}$ and so it is smooth. If $a_l=2$ then $v$ lies on $F$, $R_l$, and $Z_{i,j}$ for $(i,j)\in \mathcal I\setminus \{(l,1)\}$. Unless $k=2$, $b_s=1$ where $s\neq l$, these are $n+1$ facets and so $v$ is not smooth. In the exceptional case one of $Z_{i,j}$ does not form a facet and in this case $v$ is smooth. This proves the exceptional case (2).
\kdow

In the next example we show how information from Lemma~\ref{vSLSV} and Lemma~\ref{sSLSV} can be combined. 
\begin{exm}
	(1) Suppose $k=4$ and $\mathbf{a}=(1,1,1,1), \mathbf{b}=(1,1,1,2)$. Then the cones corresponding to the vertices $v_1=e_{1,1}+e_{2,1}$ and $v_2=e_{3,1}+e_{4,1}$ are $\sigma_{v_1}=\conv(e_{1,1}+e_{2,1}+e_{3,1}+e_{4,1}+e_{4,2}, -e_{1,1},-e_{2,1},e_{3,1},e_{4,1},e_{4,2})$ and $\sigma_{v_2}=\conv(e_{1,1}+e_{2,1}+e_{3,1}+e_{4,1}+e_{4,2},e_{1,1},e_{2,1},-e_{3,1},e_{4,2},-e_{4,1}-e_{4,2})$. Both $\sigma_{v_1}$ and $\sigma_{v_2}$ are non-smooth as expected by Lemma~\ref{sSLSV}. 
	
	(2) Suppose $k=4$ and $\mathbf a=(1,2,3,4), \mathbf{b}=(1,2,1,1)$. Here we have only two non-smooth vertices $v_1=2e_{2,1}$ and $v_2=2e_{2,2}$. The corresponding maximal singular cones are $\sigma_{v_1}=\conv(e_{1,1}+e_{2,1}+e_{2,2}+e_{3,1}+e_{4,1}, e_{1,1},-e_{2,1}-e_{2,2},e_{2,2},e_{3,1},e_{4,1})$ and $\sigma_{v_2}=\conv(e_{1,1}+e_{2,1}+e_{2,2}+e_{3,1}+e_{4,1},e_{1,1},e_{2,1},-e_{2,1}-e_{2,2},e_{3,1},e_{4,1})$, but the minimal singular cone is $\sigma_{v_1}\cap\sigma_{v_2}=\conv(e_{1,1}+e_{2,1}+e_{2,2}+e_{3,1}+e_{4,1},e_{1,1},-e_{2,1}-e_{2,2},e_{3,1},e_{4,1})$.
	
	(3) Suppose $k=3$ and $\mathbf{a}=(1,2,3), \mathbf{b}=(1,1,1)$. We have only one singular vertex $v=2e_{2,1}$ and the corresponding singular cone is $\sigma_{v}=\conv(e_{1,1}+e_{2,1}+e_{3,1}, e_{1,1},-e_{2,1},e_{3,1})$.
\end{exm}

The smoothness information for vertices of $P$ gives information about the singular locus of the projective variety $V_P$ associated to $P$. As we argued earlier, this translates to the singular locus of $\widehat T_\Delta$. By  \cite[Proposition~11.1.2]{cox2011toric}, the singular locus $V_{\rm sing}$ of $V_P$ satisfies
\[
V_{\rm sing}=\bigcup_{\sigma\, not\, smooth} V(\sigma),
\]
where $V(\sigma)=\overline{O(\sigma)}$ is the closure of the torus orbit corresponding to a cone $\sigma$ in the normal fan of $P$; see \cite[Chapter 3.2]{cox2011toric} or \cite[Section 5.1]{michalek2018selected}. Note that it is enough to find minimal cones that are singular. Indeed, if we have two cones contained in each other, then the variety corresponding to the bigger cone is contained in the variety corresponding to the smaller. Moreover, if a cone is singular, then so is every cone that contains it (if some set of ray generators cannot be completed to a basis, then no strictly larger set can). The minimal singular cones correspond to irreducible components of the singular locus.  


\begin{exm}
Let	$k=4, \mathbf a=(1,1,1,1), \mathbf b=(1,1,1,1)$. We have $\binom{4}{2}=6$ non-smooth vertices $e_{1,1}+e_{2,1}$, $e_{1,1}+e_{3,1}$, $e_{1,1}+e_{4,1}$, $e_{2,1}+e_{3,1}$, $e_{2,1}+e_{4,1}$, $e_{3,1}+e_{4,1}$. For example, the cone corresponding to $v_1=e_{1,1}+e_{2,1}$ is $\sigma_{v_1}=\conv(e_{1,1}+e_{2,1}+e_{3,1}+e_{4,1},-e_{1,1},-e_{2,1},e_{3,1},e_{4,1})$ and the cone corresponding to $v_2=e_{1,1}+e_{3,1}$ is $\sigma_{v_2}=\conv(e_{1,1}+e_{2,1}+e_{3,1}+e_{4,1},-e_{1,1},e_{2,1},-e_{3,1},e_{4,1})$. Both are singular. Now $\sigma_{v_1}\cap\sigma_{v_2}=\conv(e_{1,1}+e_{2,1}+e_{3,1}+e_{4,1},-e_{1,1},e_{4,1})$ is smooth and so we have two components of the singular locus corresponding to non-smooth cones $\sigma_{v_1}$ and $\sigma_{v_2}$. In fact we have exactly six components of the singular locus since intersection of any two corresponding cones is smooth; see  \cite[Proposition~5.5]{MOZ} for a more general case $\mathbf a=(1,1,\ldots,1), \mathbf b=(1,1,\ldots1)$.
\end{exm}

\begin{exm}
Let $\mathbf a=(1,1,1,1), \mathbf b=(1,1,1,2)$. By Lemma~\ref{sSLSV}, the singular vertices are the nine points $e_{1,1}+e_{2,1}$, $e_{1,1}+e_{3,1}$, $e_{1,1}+e_{4,1}$, $e_{1,1}+e_{4,2}$, $e_{2,1}+e_{3,1}$, $e_{2,1}+e_{4,1}$, $e_{2,1}+e_{4,2}$, $e_{3,1}+e_{4,1}$, $e_{3,1}+e_{4,2}$. This gives nine singular maximal cones $\sigma_{12}$, $\sigma_{13}$, $\sigma_{14}$, $\sigma_{15}$, $\sigma_{23}$, $\sigma_{24}$, $\sigma_{25}$, $\sigma_{34}$, $\sigma_{35}$. However, there are only six minimal cones: $\sigma_{12},\sigma_{13},\sigma_{23},\sigma_{145}=\sigma_{14}\cap\sigma_{15},\sigma_{245}=\sigma_{24}\cap\sigma_{25}, \sigma_{345}=\sigma_{34}\cap\sigma_{35}$. This can be verified by explicit computations in \texttt{Macaulay2}. The rays of the dual cone $\sigma_{12}^\vee$ lie in the product $\{(0,0), (-1,0),(0,-1)\}\times \{(1,0,0),(0,1,0),(0,0,1)\}$ -- in particular, there are nine rays in $\sigma_{12}^\vee$. The cone $\sigma_{12}$ corresponds to  $\Sec(\P^1\times\P^1)\times \P^1\times \P^2$. The rays of the dual cone $\sigma_{345}^\vee$ lie in the product $\{(1,0),(0,1)\}\times \{(0,0,0),(-1,0,0),(0,0,-1)\}$ and $\sigma_{345}$ corresponds to $\P^1\times \P^1\times \Sec(\P^1\times\P^2)$. 
\end{exm}

\begin{exm}
	Let $\mathbf{a}=(1,1,2), \mathbf{b}=(1,1,1)$. Note that $Z_{3,1}$ is not a facet of the polytope, c.f. Proposition~\ref{pSVF} (E1). The polytope has four vertices $e_{1,1}+e_{2,1}, e_{1,1}+e_{3,1}, e_{2,1}+e_{3,1}$, and $2e_{3,1}$. The corresponding maximal cones are $\sigma_{12}=\conv(e_{1,1}+e_{2,1}+e_{3,1}, -e_{1,1}, -e_{2,1}), \sigma_{13}=\conv(e_{1,1}+e_{2,1}+e_{3,1}, -e_{1,1}, e_{2,1}), \sigma_{13}=\conv(e_{1,1}+e_{2,1}+e_{3,1}, e_{1,1}, -e_{2,1})$, and $\sigma_3=\conv(e_{1,1}+e_{2,1}+e_{3,1}, e_{1,1}, e_{2,1}, -e_{3,1})$ respectively. Only $\sigma_3$ is singular, see Theorem~\ref{SLSV} (2) below for a generalization.
\end{exm}


 The next result provides a description of the singular locus of the variety associated to the polytope $P$. Hence, it can be used to completely describe the singular locus of the secant variety of any Segre-Veronese variety.

\begin{thm}\label{SLSV}
The projective variety $V_P$ is smooth if 
\begin{itemize}
\item [(S1)] $k=1$.
\item [(S2)] $k=2$, except for $a_{i_1}=2$ and $b_{i_2}>1$ $(i_1\not=i_2)$.
\item [(S3)] $k=3$ and either (i) $a_1\geq 3$, or (ii) $a_1=1$ and $a_2\geq 3$, or (iii) $a_1=a_2=a_3=b_1=b_2=b_3=1$, or (iv) $a_1=a_2=b_3=1$ and  $a_3\geq 3$.
\item [(S4)] 
	$k\geq4$, except either (i) $a_{i_1}=a_{i_2}=1$ for a pair $i_1\not=i_2$, or (ii) $a_{i_0}=2$ for $1\leq i_0\leq k$.
\end{itemize}
If $V_P$ is not smooth, its singular locus:
	\begin{itemize}
		\item [(1)] for $k=2$, has either (i) only one component if $a_{i_1}=2, a_{i_2}\neq 2, b_{i_2}>1$, $i_1\neq i_2$ or (ii) $2-s$ components if $a_1=a_2=2$, where $s$ is the number of $b_i$'s equal to $1$;
		\item [(2)] for $k=3$, has $\binom{k_1}{2}+k_2-s$ components, where $s$ is the number of three element sets $\{i_1,i_2,i_3\}$ satisfying $a_{i_1}=a_{i_2}=b_{i_3}=1$, and $k_j:=|\{i:~a_i=j\}|$;
		\item [(3)] for $k\geq4$, has $\binom{k_1}{2}+k_2$ components,	where $k_j$ are as above.
	\end{itemize}
\end{thm}
\dow 
Let $\Sigma$ be the normal fan of the polytope $P$. The vectors $u_{F}=\sum_{i,j}e_{i,j}$,  $u_{R_i}=\sum_{j=1}^{b_i}-e_{i,j}$ and $u_{Z_{i,j}}=e_{i,j}$ are the possible rays of $\Sigma$, c.f. Proposition~\ref{pSVF}. We will determine the (minimal) singular cones of $\Sigma$. Every singular cone must contain $u_F$ since the vertices which do not lie on the corresponding facet $F$ are smooth. 

Case I. $k\geq4$

In this case all the above mentioned vectors are rays of $\Sigma$, c.f. Proposition~\ref{pSVF}. By Lemma~\ref{sSLSV}, singular cones of $\Sigma$ appear when either $a_{i_1}=a_{i_2}=1$ for $i_1\not=i_2$, or $a_{i_0}=2$. That is, in all other cases the associated variety is smooth, giving (S4).

For the proof of (3), note that at least three rays of type $u_{R_i}$ do not belong simultaneously to a cone of $\Sigma$ containing $u_F$. Indeed, the intersection of corresponding facets is empty, and hence they do not form a cone of $\Sigma$. Therefore, we have three cases to discuss: (i) when a cone contains $u_F$ and exactly two rays $u_{R_{i_1}}, u_{R_{i_2}}$, (ii) when a cone contains $u_F$ and exactly one ray $u_{R_{i_0}}$, and (iii) when a cone contains $u_F$ and no ray of type $u_{R_{i}}$. In case (i), if one of $a_{i_1}$ or $a_{i_2}$ is at least $2$, then again the intersection of corresponding facets is empty and hence we do not have a cone of $\Sigma$. So assume $a_{i_1}=a_{i_2}=1$. Such a cone also contains every ray $u_{Z_{i,j}}, i\not=i_1,i_2$, since $F\cap R_{i_1}\cap R_{i_2}$ is contained in every $Z_{i,j}, i\not=i_1,i_2$, making it a singular cone of $\Sigma$. Note that any subcone of such a cone is smooth and hence we have $\binom{k_1}{2}$ components of the singular locus of the variety associated to $P$ in this case.

In case (ii), if $a_{i_0}\geq3,$ then by the same argument as above $u_F$ and $u_{R_{i_0}}$ do not form a cone of $\Sigma$. Therefore $a_{i_0}\leq2$. If $a_{i_0}=2$, then every $u_{Z_{i,j}}, i\not=i_0$, belongs to such a cone of $\Sigma$, making it a singular cone.  Note that any subcone of such a cone is smooth. If $a_{i_0}=1,$ then we will show that there is no singular cone. Indeed, such a cone cannot contain every $u_{Z_{i_0,j}}$ since intersection of $R_{i_0}$ with every $Z_{i_0,j}$ is empty. Hence such a cone, in order to be singular, would have to contain every ray $u_{Z_{i,j}},i\not=i_0$, but then it would not be a cone in $\Sigma$. Hence, we have $k_2$ components of the singular locus  of the variety associated to $P$ in this case. In case (iii), such a cone, in order to be singular, would have to contain every ray $u_{Z_{i,j}}$, but then it would not be a cone in $\Sigma$. This finishes (3) and hence the Case I.

Case II. $k=3$

We first prove (S3). Recall that $a_1\leq a_2\leq a_3$ and the vertices which are outside of $F$ are smooth. Here we can have three possibilities for $a_1$. If $a_1\geq3$, then $a_2, a_3\geq3$ which implies that the vertices of $P$ lying on $F$ are only of type $2e_{i,j}$, c.f. Lemma~\ref{vSLSV}. By Lemma~\ref{sSLSV}, these vertices are smooth and hence so is the associated variety, giving part (i) of (S3). If $a_1=2$, then again we only have vertices of type $2e_{i,j}$ lying on $F$. However, the vertices $2e_{1,j}$ for $1\leq j\leq b_1$ are singular by Lemma~\ref{sSLSV}. 
Finally, if $a_1=1$, then we consider following cases. If $a_1=1$ and $a_2\geq3$, then vertices of $P$ lying on $F$ are of type $e_{1,j}+e_{i,j'}, i\neq 1,$ and $2e_{i,j}, i\neq1$, c.f.~Lemma~\ref{vSLSV}. By Lemma~\ref{sSLSV}, all these vertices are smooth and hence so is the associated variety, giving part (ii) of (S3). If $a_1=1$ and $a_2=2$, then by similar argument as above, vertices of type $2e_{2,j}$ are singular, which gives a part of (S3). The only remaining case to consider is $a_1=a_2=1$. If $a_3\geq 3$, then the vertices of $P$ of the type $e_{1,j}+e_{3,j'}, e_{2,j}+e_{3,j'}$ and $2e_{3,j}$ are always smooth, c.f. Lemma~\ref{sSLSV}. Moreover, the vertices $e_{1,j}+e_{2,j'}$ are singular unless $b_3=1$. This confirms part (iv) of (S3). If $a_1=a_2=1$ and $a_3=2$, then 
 the vertices of $P$ of the type $2e_{3,j}$ are singular, which gives a part of (S3). If $a_1=a_2=a_3=1$, then by Lemma~\ref{vSLSV} the vertices of $P$ lying on $F$ are only of the type $e_{i,j}+e_{i',j'} (i\neq i')$, which are smooth unless $b_i\geq2$ for some $i\in \{1,2,3\}$, giving part (iii) of (S3). This finishes (S3).

For the proof of (2), we follow the same line of reasoning as in (3). That is, we again have three cases to discuss here. In case (i), we get $a_{i_1}=a_{i_2}=1$ and hence $\binom{k_1}{2}$ components of the singular locus of the variety associated to $P$ unless $b_{i_3}=1$. Indeed, $Z_{i_3,1}$ is not a facet (c.f. (E1) of Proposition~\ref{pSVF}) and hence $u_{Z_{i_3,1}}$ is not a ray of the cone of $\Sigma$. Therefore, the latter cone of $\Sigma$ is smooth as the rays of it are a part of a basis of the lattice. In case (ii), we get $a_{i_0}\leq2$. If $a_{i_0}=1,$ then there is no singular cone in $\Sigma$ as in Case I. If $a_{i_0}=2$, then we have $k_2$ components of the singular locus of the variety associated to $P$. In case (iii), again there is no singular cone in $\Sigma$ as in Case I. This finishes (2) and hence the Case II. 

Case III. $k=2$

We first prove (S2). Note that the vertices of $P$ of the type $e_{i,j}+e_{i',j'} (i\neq i')$ are always smooth, see Lemma~\ref{sSLSV}. Hence the (possible) singular vertices of $P$ lying on $F$ are of the type $2e_{i,j}$. If $a_1\geq3$, then $a_2\geq3$ and by Lemma~\ref{sSLSV} the vertices $2e_{i,j}$ are smooth, giving a part of (S2). If $a_1=2$ and $a_2\geq3$, then the vertices $2e_{1,j}$ are singular unless $b_2=1$ and $2e_{2,j}$ are smooth, c.f. Lemma~\ref{sSLSV}, giving a part of (S2). If $a_1=a_2=2$, then the vertices $2e_{1,j}$ and $2e_{2,j}$ are singular unless $b_1=b_2=1$, again confirming a part of (S2). If $a_1=1$ and $a_2=2$, then the vertices $2e_{2,j}$ are singular unless $b_1=1$, giving a part of (S2). If $a_1=1$ and $a_2\geq3$, then by Lemma~\ref{sSLSV} the vertices $2e_{2,j}$ are smooth. This finishes (S2).

We now prove (1). The case $a_1=a_2=1$ gives smooth variety associated to the polytope $P$ since the cones in $\Sigma$ correspond to vertices of type $e_{i,j}+e_{i',j'} (i\neq i')$ which are smooth as discussed above. For the rest of the proof we follow the same line of reasoning as in Case I. We, in fact, have only two cases (ii) and (iii) to discuss. In case (ii), the cone contains either $u_F$ and $u_{R_1}$ or $u_F$ and $u_{R_2}$. If it contains $u_F$ and $u_{R_1}$, then $a_1\leq2$. If $a_1=1$, we proceed as above and get only smooth cones (if exist) in $\Sigma$. If $a_1=2$, then every ray $u_{Z_{2,j}}$ belongs to the cone as well, making it a (minimal) singular cone of $\Sigma$, unless $b_2=1$. This gives one component of the singular locus of the variety associated to $P$, unless $b_2=1$. In the latter case $Z_{2,1}$ is not a facet of $P$ (c.f. (E3) of Proposition~\ref{pSVF}) and hence such a cone of $\Sigma$ is smooth. If now the cone contains $u_F$ and $u_{R_2}$, then $a_2\leq2$. The case $a_2=1$ has already been discussed above, as $a_1=1$ in this case. If $a_2=2$, then every ray $u_{Z_{1,j}}$ belongs to the cone as well, making it a singular cone of $\Sigma$, unless $b_1=1$. This gives one component of the singular locus, unless $b_1=1$. In the latter case $Z_{1,1}$ is not a facet of the polytope $P$ (c.f. (E2) of Proposition~\ref{pSVF}) and hence such a cone of $\Sigma$ is smooth. Note that if $a_1=a_2=2$, then there are two singular cones in $\Sigma$, one containing $u_F$ and $u_{R_1}$ and other containing $u_F$ and $u_{R_2}$, unless $b_1=1$ or $b_2=1$. In case (iii), we follow the same line of reasoning as in Case I. This finishes (1) and hence the case Case III.

Case IV. $k=1$

All the cones of $\Sigma$ are smooth since all the vertices of the polytope $P$ are smooth, c.f. Lemma~\ref{sSLSV}. Therefore, the associated variety to the polytope $P$ is smooth. This completes (S1) and hence the proof.
\kdow
The cases when the polytope $P$ is smooth correspond to smoothness of $V_P$. This projective variety can be smooth if and only if either:
\begin{enumerate}
\item $V_P$ is a projective space, which happens if and only if $\widehat V_P$ is smooth and if and only if all lattice points of $P$ are linearly independent;
\item zero is the unique singular point of $\widehat V_P$.
\end{enumerate}

In the first case the secant variety is smooth, which happens if and only if it fills the ambient space. These cases were characterized in Theorem~\ref{GSV}. Recall that the zero point in $\widehat V_P$ corresponds to the Segre-Veronese variety. This means that the second case happens if and only if the singular locus of the secant variety coincides with the Segre-Veronese variety.
\begin{cor}
The singular locus of $\sigma_2(X)$ coincides with $X$ if and only if:
\begin{itemize}
\item $k=1$ and $a>2$ or ($a=2$ and $b>1$);
\item $k=2$ and ($a_1,a_2>2$) or ($a_1=a_2=1$ and $b_1,b_2>1$) or ($a_{1}=2$, $b_{2}=1$ and $a_2\neq 2$) or ($a_2=2$, $b_1=1$ and $a_1\neq 2$) or ($a_1=a_2=2$ and $b_1=b_2=1$);
\item $k=3$ and $a_1\geq 3$ or ($a_1=1$ and $a_2\geq 3$) or ($a_1=a_2=b_3=1$ and  $a_3\geq 3$);
\item $k\geq 4$ and all $a_i\neq 2$ and there is at most one $a_i=1$. 
\end{itemize}
\end{cor}

We are now ready to provide the description of the singular locus of the secant variety of Serge-Veronese variety. 
\begin{cor}\label{cSLSV}
If $a_{i_1}=a_{i_2}=1$ then the corresponding component of the singular locus of $\sigma_2(X)$ is isomorphic to
		\[
		v_{a_1}(\P^{b_1})\times\dots\times\widehat{v_{a_{i_1}}(\P^{b_{i_1}})}\times\dots\times\widehat{v_{a_{i_2}}(\P^{b_{i_2}})}\times\dots\times v_{a_k}(\P^{b_k})\times \Sec(\P^{b_{i_1}}\times \P^{b_{i_2}}),
		\]
		where $\widehat{\cdot}$ denotes omission.
If $a_i=2$ then the corresponding component of the singular locus of $\sigma_2(X)$ is isomorphic to
		\[
		v_{a_1}(\P^{b_1})\times\dots\times\widehat{v_{a_{i}}(\P^{b_{i}})}\times\dots\times v_{a_k}(\P^{b_k})\times \Sec(v_{a_{i}}(\P^{b_{i}})).
		\]
		If $k\geq 4$ then these are precisely all components of the singular locus. 
\end{cor}
\begin{proof}
By the proof of Theorem~\ref{SLSV} the face representing the component of the singular locus of $V_P$ corresponding to $a_{i_1}=a_{i_2}=1$ has lattice points $e_{i_1,j_1}+e_{i_2,j_2}$. For $a_i=2$ the corresponding lattice points are of the form $e_{i,j_1}+e_{i,j_2}$ and $2e_{i,j}$. These respectively parameterize $\Sec(\P^{b_{i_1}}\times \P^{b_{i_2}})$ and $\Sec(v_{a_{i}}(\P^{b_{i}}))$. Recalling that the zero point in the affine space corresponds to the Segre-Veronese variety we obtain the claimed result.
\end{proof}

\section*{Acknowledgments.} AK is thankful to MPI Leipzig for the invitation and financial support for a visit to the institute during this project. He is also grateful to the non-academic staff, in particular Saskia Gutzschebauch, at MPI for helping him to have a very smooth arrival and then a nice stay.

\appendix

%
%
%
%


\bibliographystyle{siam}
\bibliography{toriccum}

\end{document}